\documentclass[final, reqno, 11pt]{article}
\usepackage[margin=2.5cm]{geometry}
\usepackage{amsmath,textcomp,latexsym,graphics}
\usepackage{amssymb,amsthm}

\DeclareGraphicsExtensions{.pdf,.png,.jpg}

\usepackage[usenames]{color}

\numberwithin{equation}{section}

\newcommand{\ds}{\displaystyle}
\newcommand{\ol}{\overline}

\newcommand{\be}{\begin{equation}}
\newcommand{\ee}{\end{equation}}
\newcommand{\ba}{\begin{array}}
\newcommand{\ea}{\end{array}}

\newcommand{\bpm}{\begin{pmatrix}}
\newcommand{\epm}{\end{pmatrix}}

\newcommand{\Tr}{\textrm{Tr}}

\newcommand{\bd}{\begin{definition}}
\newcommand{\ed}{\end{definition}}

\newcommand{\np}{\nabla^+}
\newcommand{\nm}{\nabla^-}
\newcommand{\phip}{\phi^+}
\newcommand{\phim}{\phi^-}
\newcommand{\hf}{\hat{f}}
\newcommand{\Pip}{\Pi^+}
\newcommand{\Pim}{\Pi^-}

\newcommand{\wtV}{\widetilde{V}}

\newenvironment{proofsk}{\bigskip\noindent{\it Proof sketch.}\rm}{\hfill $\Box$}

\newtheorem{theorem}{Theorem}[section]
\newtheorem{proposition}[theorem]{Proposition}
\newtheorem{lemma}[theorem]{Lemma}
\newtheorem{cor}[theorem]{Corollary}

\theoremstyle{remark}
\newtheorem{remark}{Remark}

\theoremstyle{definition}
\newtheorem{definition}{Definition}[section]

\newtheorem{assumptions}{Assumptions}[section]

\newtheorem{dichotomy}{Dichotomy}[section]
\newtheorem{example}{Example}[section]
\let\cal=\mathcal

\begin{document}

\title{On the behavior at infinity of solutions \\ to difference equations in Schr\"odinger form}
\author{Evans M. Harrell II\\
School of Mathematics\\
Georgia Institute of Technology\\
Atlanta, GA 
30332-0610 USA\\
email: harrell@math.gatech.edu
\and
Manwah Lilian Wong\\
School of Mathematics\\
Georgia Institute of Technology\\
Atlanta, GA 
30332-0610 USA\\
email: wong@math.gatech.edu}

\maketitle
\begin{abstract}
We offer several perspectives on the behavior at infinity of solutions
of discrete Schr\"odinger equations.
First we study pairs of discrete
Schr\"odinger equations whose potential functions differ by a 
quantity that can be considered small 
in a suitable sense
as the index $n \to \infty$. 
With simple assumptions on the growth rate of the solutions of the original system, 
we show that the perturbed system has a fundamental set of solutions with the 
same behavior at infinity, employing a variation-of-constants 
scheme to produce a convergent 
iteration for the solutions of the second equation in terms of those of the original
one. We use the 
relations between the solution sets 
to derive exponential dichotomy 
of solutions and elucidate the structure of transfer matrices.

Later, we present a sharp discrete analogue 
of the Liouville-Green (WKB) transformation, making it possible to 
derive exponential behavior at infinity of a single difference equation,
by explicitly constructing a comparison equation 
to which our perturbation results apply.  
In addition, we point out an exact relationship 
connecting the diagonal part of the Green matrix to the asymptotic behavior of solutions.
With both of these tools it is possible to identify an Agmon metric, in terms of which, in some situations, any decreasing solution must decrease exponentially.

A discussion of the discrete Schr\"ordinger problem and its connection with orthogonal polynomials on the real line 
is presented in an Appendix.

\end{abstract}


\date{}

\maketitle

\section{Introduction}\label{intro}

In this article we address the asymptotic behavior of solutions to linear difference equations of Schr\"odinger type, as the index $n$ tends to infinity. We prove exponential dichotomy theorems and refined approximative expressions for the growing and subdominant (i.e., decaying) solutions, which have controlled errors.

We begin by approaching the subject as a perturbation analysis, showing that if two Schr\"odinger difference equations have potential terms that are sufficiently close, then they are asymptotically equivalent in the sense of \cite{Eas}, that is, there are solution bases for the two problems with the same behavior at infinity. 
The expressions obtained by the perturbation analysis
are not merely asymptotic, but convergent for large but 
finite indices $n$.
We follow with 
a classification of the possible asymptotic behaviors
and some more estimates, 
including some cases where the asymptotic behavior of solutions 
does not match that of the comparison equation but
can nonetheless be characterized.

Of course, when faced with one particular equation, 
comparison theorems are of limited use in the absence of a good
equation to which one can compare.  We therefore present some methods for 
constructing such equations after the perturbation analysis.  
Finally, we
present some examples and remarks about connections with 
orthogonal polynomials.

Let $\Delta$ denote the discrete second-difference operator on 
the positive integer lattice.  We standardize the Laplacian such that
$(\Delta f)_n := f_{n+1} + f_{n-1} - 2 f_n$ for $f = (f_n) \in \ell^2(\mathbb{N})$,
and consider pairs of equations of the form 
\begin{align}
(-\Delta+V)\psi & =  0 , \label{one*} \\ 
(-\Delta+V^0)\phi& =  0, \label{compar}
\end{align}
where the potential-energy 
functions $V$, $V^0$ are diagonal operators with 
real values $V_n$ and $V^0_n$ respectively. 
(Complex $V_n$ and $V^0_n$ 
could be allowed with, for the most part, only straightforward complications, but we prefer to keep the exposition focused.)

Our first aim is to  
find conditions under which the solutions of \eqref{one*} have the
same asymptotic behavior as $n \to \infty$ as those of the comparison \eqref{compar}
when the potential energies 
$V$ and $V_0$ are close in a suitable sense.
One application of the analysis is to the 
asymptotic behavior of eigenfunctions, in which case 
instead of \eqref{one*} {one could}
write
\begin{equation}\label{EVP}
(-\Delta +V - E)\psi=0
\end{equation}
and $(-\Delta +V ^0- E)\phi=0$
for some real eigenvalue $E$.
Again, for simplicity
we shall absorb $E$ into the definition of $V$, 
with no material restriction, because
we consider the full set of solutions to
\eqref{one*} without restricting to 
eigensolutions of a particular realization of 
$-\Delta+V$ as an operator.  Those interested in decay properties of  
eigenfunctions should systematically replace $V_n$ in this article
by $V_n - E$.

We do assume, however, that among
the solution set of the comparison
equation \eqref{compar} there is a distinguished solution
that decreases at infinity, unique up to a multiplicative constant, and
we follow the nomenclature of ordinary differential equations in
referring to such solutions as {\it subdominant}.
(The term {\it recessive} is also frequently used.)  We recall at this stage
that if $V$ has a constant value $V_\infty \notin [-4,0]$, then
explicit solutions are easily found, and it emerges that
$(-\Delta - V_\infty)\phi=0$ has a subdominant solution, indeed, 
one that decreases exponentially 
(see Example \ref{constantpotentialexample}). Conversely, if 
${V =} V_\infty \in [-4,0]$, then there are no subdominant
solutions.  The significance of the interval $[-4,0]$ is that it is the 
spectrum of $\Delta$.  

\begin{remark}\label{4 vs 0}
Equation
\eqref{one*} is invariant under the 
transformation
\begin{eqnarray}\label{sym}
\psi_n \to (-1)^n \psi_n\\
V_n \to - 4 - V_n,
\end{eqnarray}
as can be easily verified.
Because of this, any fact proved under the assumption, for example, that $V_n > 0$
has a counterpart for $V_n < -4$.  We shall use this remark to 
avoid repetition in some of our proofs.
\end{remark}

When \eqref{compar} has a 
subdominant solution, it will be denoted $\phi^-$ (fixing an overall constant), and ordinarily we shall
identify a second, independent solution as $\phi^+$.  We recall that the 
Wronskian of two solutions of a discrete Schr\"odinger equation,
\begin{equation}\label{Wdef} 
W[\phi_-, \phi_+] := \phi^-_n \phi^+_{n+1} - \phi^-_{n+1} \phi_n^+ ,
\end{equation}
is independent of the coordinate $n$,
analogously to 
a well-known fact for Sturm-Liouville equations.  (See, e.g., \cite{Aga}.) In terms of difference operators $\nabla^\pm$, 
\be
\np f_n :=   f_{n+1}- f_n  \text{ and } \nm f_n  :=  f_n - f_{n-1} \label{nabladef} .
\ee
the Wronskian can also be expressed as
\be
W  =  \phi^-_n (\nabla^{\pm} \phi_n^+) - \phi_n^+ (\nabla^{\pm} \phi^-_n ).
\label{wronskianpm}
\ee
For future reference we recall some simple relations for the difference operators:
\begin{enumerate}
\item $\np (\nm f_n) = \nm (\np f_n) = \Delta f_n$;
\item (Chain Rule) $\np (f g)_n = (\np f_n) g_n + f_n (\np g_n) + (\np f_n ) (\np g_n)$;
\item  (Chain Rule) $\nm (f g)_n = (\nm f_n) g_n + f_n (\nm g_n) - (\nm f_n ) (\nm g_n)$.
\end{enumerate}

In comparing \eqref{one*} and \eqref{compar} the behavior of solutions at infinity will be examined from several points of view. First, we
study the asymptotic behavior of solutions under the effect of small perturbations of the potential as a fixed-point problem.  We consider the solutions of  \eqref{compar} as known, and use them as the basis for a (convergent) 
variation-of-constants calculation of the solutions of \eqref{one*}.  
Then we introduce a factorization of the equation satisfied by the coefficients in that scheme in order to get a detailed understanding of how they 
converge.  

Thereafter we present a new and efficient discrete variant of the Liouville-Green (WKB) 
approximation \cite{Olv}, so that for a given potential $V$
a comparison equation \eqref{compar} can be found 
for which the asymptotics are explicitly known,
and consequently the behavior at infinity of solutions of \eqref{one*}
is explicitly determined, with controlled errors. 
As an alternative, following \cite{DaHa,ChSh00,ChSh01}, we explore a set 
of related exact relations based on the diagonal of the Green 
matrix and their consequences for the behavior of solutions at infinity.

Finally, the reader may refer to the Appendix for a discussion of the relation between orthogonal polynomials and second-order difference equations. There the connection between ratio asymptotics of orthogonal polynomials and the results of Geronimo--Smith \cite{GeSm} will be discussed, and it will be shown how
solutions of the discrete Schr\"odinger equation can be represented by orthogonal polynomials of the first and second kind.

We are far from the first to consider these questions, and like other researchers
we mimic the better-developed theory known for Sturm-Liouville problems.
Let us close the Introduction by placing our work in the context of the earlier literature.

A systematic study of certain difference equations dates from Poincar\'e \cite{Poi}.  In his work and in that of Birkhoff \cite{Bir} asymptotic analysis was considered for equations using what would nowadays be termed transfer matrices of special types.  
A rather satisfactory understanding of the effect of small perturbations on stability questions for equations using transfer matrices,
with dichotomy assumptions on their eigenvalues,
was developed in \cite{Per,Cof,BeLu}, some of which is recounted in
the monograph by Agarwal \cite{Aga},
which is a good source for showing how 
many of the standard facts from Sturm-Liouville theory 
can be ported over to the discrete setting,  
in particular, the technique of variation of constants.
Coffman \cite{Cof} and Benzaid--Lutz \cite{BeLu}
studied product solutions, and in that regard prefigure
in a rough way what we do in Section \ref{WKB}. 
The main results of \cite{BeLu} were discrete analogues of 
Levinson's fundamental lemma \cite{Lev} 
for the asymptotic expression of the solution of a perturbed linear differential equation.
In \cite{BeLu} the authors considered difference equations
of the form
\be
y(k+1)=[\Lambda(k) + R(k)] y(k) .
\label{ble}
\ee 
Here $\Lambda(k)$ is an $N \times N$ diagonal matrix with non-zero diagonal entries $(\lambda_j (k))_{j=1}^{N}$ 
that satisfy a certain 
dichotomy condition.
They further considered
\be
x(k+1) = [\Lambda_0 + V(k) + R(k)] x(k) ,
\ee 
again where $\Lambda_0$ is diagonal and $V$ and $R$ satisfy certain bounds.
Their results apply widely to perturbed difference equations, but not readily to \eqref{one*} and \eqref{compar}:  As we shall see in \eqref{e9},
the transfer matrices in the present article are of the form
\be
I + M_n =  I + \ds \frac{V_{n}-V^0_n}{W} \bpm \phi^+_n \phi^-_n & \phi^-_n \phi^-_n \\ - \phi^+_n \phi^+_n & - \phi^+_n \phi^-_n \epm ,
\ee
which are neither diagonal nor diagonable if $V_n - V^0_n \neq 0$.
In fact, $1$ is the only eigenvalue of the matrix $I+M_n$ and it has geometric multiplicity one.  Moreover, the term $\frac{V_n - V^0_n}{W} (\phip_n)^2$ in the lower left corner typically diverges as $n \to \infty$.

Trench \cite{Tr1,Tr2} succeeded in giving conditions for the asymptotic equivalence of the solution sets of \eqref{one*} and \eqref{compar} in the sense considered by Hartman and Wintner \cite{Hart}, and seems to have been the first to realize that a good criterion for equivalence relies on an analysis of the expression
\be\label{TrenchJ}
J_k := \phi_k^+ \phi_k^- \left(V_k - V_k^0\right),
\ee
(in our notation).  In \cite{ChSh08}, following Trench,
a necessary and sufficient condition for asymptotic equivalence for some difference equations related to \eqref{one*} and \eqref{compar} is spelled out in terms of $J$.
Although we bring different methods to bear on 
asymptotic equivalence in the following sections, 
$\ell^p$ norms of \eqref{TrenchJ} and similar quantities
remain central; see Theorems \ref{Banconv}, \ref{dichothm1}, and \ref{dichothm2}.
One could interpret these norms as traces of operator perturbations like those 
occurring in studies of spectral--shift functions 
(e.g., see \cite{GeSi}),
leading us to speculate that direct connections
between the spectral-shift functions and behavior at infinity could be found. 

After a discussion of asymptotic equivalence, we take advantage of the specific 
Schr\"odinger form of the equation, and
construct comparison equations having
product solutions of a certain structure, inspired by the classical
Liouville-Green, or WKB, approximation.
Of prior work
on discrete versions of the Liouville-Green approximation
we single out that of
Geronimo and Smith \cite{GeSm}, which was inspired by 
some earlier work of Braun \cite{Bra}.
Geronimo and Smith studied a somewhat more general equation than \eqref{one*},
\be
d_{n+1} y_{n+1} - q_n y_n + y_{n-1} = 0,
\label{jeff1}
\ee where $d_n$ and $q_n$ are sequences of numbers with $d_n \ne 0$ for $n=1,2, \dots$, and pursued a Riccati analysis for solutions in product form.  
In Section \ref{WKB} we
identify a more explicit and efficient product scheme
along the lines of the Liouville-Green approximation as presented in \cite{Olv},
to which we apply the perturbation analysis developed in Section \ref{VoC}.
Yet another 
article with Liouville-Green analysis using products is \cite{Che}, in which an explicit semiclassical parameter appears, and the Green matrix is studied in a product form and used to prove refined stability results for nonhomogeneous difference equations.
In the following subsection we relate 
the discrete Liouville-Green approximation to the
diagonal of the Green function, following
ideas pioneered in \cite{DaHa}, which have
previously been somewhat developed in the study of
difference equations by Chernyavskaya and Shuster \cite{ChSh00,ChSh01,Che}.

\section{Variation of constants and behavior at infinity}\label{VoC}

We begin by casting the problem of understanding the asymptotic dependence of solutions at infinity as a problem on a certain weighted Banach space, following 
ideas of \cite{HaSi} in the continuous case, which was in 
turn inspired by \cite{Hart}.  
Suppose that $V$ is close to another potential $V^0$ such that the solutions 
to $(-\Delta + V^0) \phi = 0$
are understood, in the sense that a pair of independent solutions
$\phi^\pm$ can be identified, 
including a subdominant solution $\phi^-_n\in\ell^2$. 
A perturbation analysis
can be based on the following way of connecting the
solutions of \eqref{one*} and \eqref{compar}.
\begin{theorem} Let $V$ and $V^0$ be two potential functions, and  
let $\phi^\pm$ be independent
solutions to the equation \eqref{compar}. We may represent any $\psi$ as a linear combination of  $\phi^\pm$ with variable 
coefficients $a_n^\pm$, i.e.,
\be \label{BasicRep}
\psi_n = a_n^+ \phi^+_n + a_n^- \phi^-_n.
\ee 
Then $\psi$ is a solution to the equation \eqref{one*} if and only if we may find sequences $(a^\pm_n)_{n=1}^\infty$ that satisfy the following two conditions:
For all $n \geq 1$,
\begin{eqnarray}
(\nm a^+_n) \phi^+_{n-1} + (\nm a^-_n) \phi^-_{n-1} =  0 ,
\label{e2c} \\
 \bpm a^+ _{n+1} \\ a^- _{n+1} \epm =
 \left[ I + \ds \frac{V_{n}-V^0_n}{W} \bpm \phi^+_n \phi^-_n & \phi^-_n \phi^-_n \\ - \phi^+_n \phi^+_n & - \phi^-_n \phi^+_n \epm \right] \bpm a^+_{n} \\ a^-_{n} \epm =: \left( I+M_n  \right) \bpm a^+_{n} \\ a^-_{n} \epm ,
\label{e9} 
\end{eqnarray} under the convention that $\phi^\pm_{0} = 0$.
\label{recurrencethm}
\end{theorem}

\begin{proof}[Proof of Theorem~{\rm\ref{recurrencethm}}] Suppose that 
$\psi$ is a solution to \eqref{one*}. Since the expression \eqref{BasicRep} has two degrees of freedom we have the liberty to impose a second condition on the coefficients to
so that
\be\label{psinexpand}
\nm \psi_n = a^+_n( \nm \phi^+_n) + a^-_n (\nm \phi^-_n).
\ee
Observe that \eqref{psinexpand} implicitly sets the following expression to zero
\be
 (\nm a^+_n) \phi^+_n  +(\nm a^-_n) \phi^-_n - (\nm a^+_n) (\nm \phi^+_n) - (\nm a^-_n) (\nm \phi^-_n) = 0;
\label{e2}
\ee
by direct expansion \eqref{e2} is equivalent to \eqref{e2c}.

Now compute $\Delta \psi_{n} = \np \nm \psi_{n}$ based on the expression \eqref{psinexpand}. By the chain rules for $\nabla^\pm$,
\begin{multline}
\Delta \psi_n = (\np a^+_n ) (\nm \phi^+_n ) + a^+_n (\Delta \phi^+_n) + (\np a^+_n) (\Delta \phi^+_n)
\\ + (\np a^-_n ) (\nm \phi^-_n ) + a^-_n (\Delta \phi^-_n)+ (\np a^-_n) (\Delta \phi^-_n).
\label{e3}
\end{multline}

\noindent
By substituting $\Delta \psi_n = V_n \psi_n$ and $\Delta \phi^{\pm}_n = V^0_n \phi^{\pm}_n$ into \eqref{e3}, we obtain
\be
(V_n-V^0_n) (a^+_n \phi^+_n + a^-_n \phi^-_n ) = (\nm \phi^+_n + V^0_n \phi^+_n) (\np a^+_n) + (\nm \phi^-_n + V^0_n \phi^-_n) (\np a^-_n) .
\label{e5}
\ee
In order for the coefficients of $\np a^-_n$ in \eqref{e5} and \eqref{e2c} to match, we multiply \eqref{e5} by $\phi^-_{n}$ and \eqref{e2c} by $(\nm \phi^-_n + V^0_n \phi^-_n)$. Then we subtract one from the other and get\be
(\np a^+_n ) W = (V_n - V^0_n) \left( \phi^+_n  a^+_n + {\phi^-_n}  a^-_n \right) \phi^-_n = (V_n - V^0_n) \psi_n \phim_n ,
\label{e6}
\ee where $W$ is the Wronskian as defined in \eqref{Wdef}.

Similarly, we match the coefficients of $\np a^-_n$ in \eqref{e5} and \eqref{e2c} by multiplying \eqref{e5} with $\phi^+_{n}$ and \eqref{e2c} with $(\nm \phi^+_n + V^0_n \phi^+_n)$. Then we obtain
\be
(\np a^-_n ) (-W) = (V_n-V^0_n) \left( \phi^+_n \phi^+_n a^+_n + \phi^-_n \phi^+_n a^-_n\right) = (V_n - V^0_n) \psi_n \phip_n.
\label{e7}
\ee 
Putting \eqref{e6} and \eqref{e7} together in matrix form, we arrive at \eqref{e9}.

For the implication in the other direction,
suppose that the sequences $a^\pm_n$ satisfy \eqref{e9} and \eqref{e2c}. By direct expansion of \eqref{BasicRep}, we find that
\be
\Delta \psi_n
 =  V^0_n \psi_n + \left(\text{right side of } \eqref{e5}\right) + \nabla^+ \left(\text{left side of } \eqref{e2c}\right).
\label{sumpsi}
\ee
With \eqref{e9}, \eqref{e6} and \eqref{e7} follow.
If we now apply these relations to the right side of \eqref{e5}, it becomes
\be
\ds \frac{V_n - V^0_n}{W}\left[ (\nm \phip_n + V^0_n \phip_n) \psi_n \phim_n - (\nabla^- \phim_n + V^0_n \phim_n) \psi_n \phip_n \right],
\ee
which simplifies to
\be
\psi_n \ds \frac{V_n - V^0_n}{W} [\phim_n \nm \phip_n - \phip_n \nm \phim_n] = (V_n - V^0_n) \psi_n .
\ee
Returning to \eqref{sumpsi}, we conclude that $\Delta \psi_n = V_n \psi_n$ for all $n$ only if 
$\np\left((\nm a^+_n) \phi^+_{n-1} + (\nm a^-_n) \phi^-_{n-1}\right) = 0$
for all $n$, or in other words, when the expression on the left side
of \eqref{e2c} is a constant. Finally, note that since its value is zero when $n=1$, it is zero for all $n$.
\end{proof}

\subsection{Convergence of $a^\pm_n$ in a suitable Banach space}

An advantage of the variation-of-constants approach to
asymptotic equivalence over the methods of \cite{Tr1,Tr1,ChSh08}
is that it provides a rapidly convergent iterative scheme with error estimates 
that can be made explicit.
To set it up,
we introduce the notation
\be\label{betandef}
{\bf a}_n =\begin{pmatrix} {a_n^+} \\ {a_n^-} \end{pmatrix}  \text{ and }
\quad  \beta_n =\frac{V_n -V^0_n} W ,
\ee  where $W$ is the Wronskian as in \eqref{Wdef}.  We shall regard $\bf a$ as an element of a weighted Banach space,
\be {\cal B}_N:=\left\{{\bf X}=\begin{pmatrix} X^+_n\\ X^-_n\end{pmatrix}: 
\|{\bf X}\|_N :=\sup_{n\ge N}
\left(|(\phi^+_n)^2 X^+_n| + |X^-_n|\right)<\infty\right\}.
\ee

\noindent
Substituting the expression for $\psi$ into \eqref{e9},
we calculate
\begin{equation}\label{two*}
(\nabla^-{\bf a}_n)=\beta_n
\begin{pmatrix}
\phi^+_n\phi^-_n & (\phi^-_n)^2 \\
-(\phi^+_n)^2 &  - \phi^+_n\phi^-_n \end{pmatrix} {\bf a}_n,
\end{equation}
or, by summing \eqref{two*},
\begin{align*}
{\bf a}_n&= {\bf a}_{n+1}-\beta_{n} \begin{pmatrix} 
\phi^+_{n}\phi^-_{n} & (\phi^-_{n})^2\\
-(\phi^+_{n})^2 & -\phi^+_{n}\phi^-_{n} \end{pmatrix} {\bf a}_{n}\\
&= \dots\\
&= {\bf a}_{n+\ell} - \sum^{n+\ell-1}_{k=n} \beta_{k} \begin{pmatrix} 
\phi^+_{k}\phi^-_{k} & (\phi^-_{k})^2\\
-(\phi^+_{k})^2 & -\phi^+_{k}\phi^-_{k} \end{pmatrix} {\bf a}_{k} .
\end{align*} 
Formally letting $\ell\to\infty$, $ a_{n+\ell}\to
\begin{pmatrix}0 \\ 1\end{pmatrix}$, we therefore 
define a linear operator ${\cal M}$ by
\be ({\cal M} {\bf X})_n:=\sum^\infty_{k=0}\beta_{n+k}
\begin{pmatrix} \phi^+_{n+k}\phi^-_{n+k} & (\phi^-_{n+k})^2\\
-(\phi^+_{n+k})^2 & - \phi^+_{n+k}\phi^-_{n+k}\end{pmatrix} {\bf X}_{n+k}.
\ee 

The convergence of the coefficients in the Banach space proceeds as follows:
\begin{theorem}\label{Banconv}
Suppose that \eqref{compar} has a solution basis
$\phi^\pm$ such that $\lim_{n \to \infty}\phi_n^- = 0$,
$|\phi_n^+|$ is monotonically nondecreasing for sufficiently large $n$, 
and $\beta_n$ (cf. \eqref{betandef}) satisfies
$\beta_n (1 + |\phi_n^+ \phi_n^-|^2) \in\ell^1$.
Then for $N$ sufficiently large,
${\cal M}$ is a contraction on
${\cal B}_N$.  Consequently, there exists a 
unique solution $\psi^-$
of \eqref{one*} such that
$$\psi_n^- = a_n^+ \phi_n^+ + a_n^- \phi_n^-,$$
where $\lim_{n \to \infty}{a_n^+} = 0$
and
$\lim_{n \to \infty}{a_n^-} = 1$.  
Moreover,
if we define $\widehat{\psi}_n^- := \max_{m \ge n}{|\phi_n^-|}$, then
\begin{equation}
\psi_n^- = \phi_n^- + r_n \widehat{\psi}_n^-,
\end{equation}
with 
$\lim_{n \to \infty}{r_n} = 0$.
\end{theorem}

\begin{remark}
If $|\phi_n^-|$ is monotone nonincreasing, then
we may simply write
$\psi_n^- = (1 + r_n) \phi_n^-$,
with 
$\lim_{n \to \infty}{r_n} = 0$.  
If the product
$\phi_n^+ \phi_n^-$ is bounded, it suffices for this theorem
to assume that $\beta_n \in \ell^1$; 
circumstances under which this
is guaranteed are discussed below in Lemma \ref{secondsol}, Theorem \ref{Snestimates1} and Theorem \ref{Snestimates2}.
\end{remark}
 
\begin{proof}
For $N$ sufficiently large, we 
claim that ${\cal M}$ is a strict contraction on 
${{\cal B}_N}$.  
To see this, we introduce the shorthand
$\sup\left|\begin{pmatrix} X^+_n\\ X^-_n\end{pmatrix}\right|$
for $\sup_{n \ge N}(|X^+_n| + |X^-_n|)$, and observe that
$\| {\bf X}\|_{{\cal B}_N} = \sup_{m\ge N}\left|
\begin{pmatrix}(\phi^+_m)^2 & 0\\ 0 & 1\end{pmatrix} {\bf X}_m\right|$.
Thus
\begin{align*}
\|{\cal M} {\bf X}\|_{{\cal B}_N}&=\sup_{n \ge N}\Biggl|\begin{pmatrix}
|\phi^+_{n}|^2 & 0\\ 0 & 1\end{pmatrix}\sum^\infty_{k=0}
\beta_{n+k}
\begin{pmatrix}\phi^+_{n+k}\phi^-_{n+k} & (\phi^-_{n+k})^2\\ 
-(\phi^+_{n+k})^2 & - \phi^+_{n+k}\phi^-_{n+k}\end{pmatrix} {\bf X}_{n+k}\Biggr|\\
&\qquad =  
\sup_{n\ge N}\Biggl|\sum_{k=0}^\infty \beta_{n+k} \begin{pmatrix}
|\phi^+_{n}|^2 & 0\\ 0 & 1\end{pmatrix}
\begin{pmatrix}\phi^+_{n+k}\phi^-_{n+k} & (\phi^-_{n+k})^2\\ 
-(\phi^+_{n+k})^2 & - \phi^+_{n+k}\phi^-_{n+k}\end{pmatrix} \times\\
&\qquad \quad \quad\quad\quad\quad
\begin{pmatrix}
|\phi^+_{n+k}|^2 & 0\\ 0 & 1\end{pmatrix}^{-1}
\begin{pmatrix}
|\phi^+_{n+k}|^2 & 0\\ 0 & 1\end{pmatrix} {\bf X}_{n+k}\Biggr|\\
&\qquad =  
\sup_{n\ge N}\Biggl|\sum_{k=0}^\infty \beta_{n+k} 
\begin{pmatrix}\phi^+_{n+k}\phi^-_{n+k}\left|\phi_n^+/\phi_{n+k}^+\right|^2 & (\phi^+_{n+k}\phi^-_{n+k})^2 \left|\phi_n^+/\phi_{n+k}^+\right|^2\\ 
- 1  & - \phi^+_{n+k}\phi^-_{n+k}\end{pmatrix} \begin{pmatrix}
|\phi^+_{n+k}|^2 & 0\\ 0 & 1\end{pmatrix} {\bf X}_{n+k}\Biggr|\\
&\qquad \le \|{\bf X}\|_{{\cal B}_N}
\sup_{n\ge N} \sum_{k=0}^\infty\left|\beta_{n+k}\right| \max(|\phi^+_{n+k}\phi^-_{n+k}|+|\phi^+_{n+k}\phi^-_{n+k}|^2, 1+ |\phi^+_{n+k}\phi^-_{n+k}|)\\
&\qquad \le 2 \|{\bf X}\|_{{\cal B}_N}
\sum_{k=0}^\infty\left|\beta_{N+k}\right|(1+|\phi^+_{N+k}\phi^-_{N+k}|^2)\\
\end{align*}

We then ask whether there is a solution to
$$
{\bf a}_n=\begin{pmatrix} 0 \\ 1\end{pmatrix} - ({\cal M} {\bf a})_n,
$$
and conclude by the contraction mapping theorem that there is, 
for $N$ large enough that 
$$\sum_{k=0}^\infty\left|\beta_{N+k}\right|(1+|\phi^+_{N+k}\phi^-_{N+k}|^2) < \frac1 2 .$$
Indeed, therefore
$({\bf 1}+{\cal M}) {\bf a}= \begin{pmatrix} 0 \\ 1\end{pmatrix}$ is 
uniquely solved by the 
norm-convergent Neumann series 
\begin{equation}\label{Neum}
{\bf a}=\left(\sum^\infty_{\ell =0}(-{\cal M})^\ell\right) \begin{pmatrix} 0 \\ 1\end{pmatrix}.
\end{equation} Being dominated by a geometric series, the convergence of 
\eqref{Neum} is exponentially fast.

For the final statement, we need a lemma about the operator
${\cal M}$:
\begin{lemma}
Suppose that $|x_n^+| \le C_1 |\phi_n^-|^2$ and 
$|x_n^-| \le C_2$.  Then
\be
|{\cal M} x|_n^+ \le M \,\,(C_1 + C_2) |\phi_n^-|^2
\ee
and
\be
|{\cal M} x|_n^- \le M \,\,(C_1 + C_2),
\ee
where $M  :=  \|\beta_n (1+|\phi_n^+ \phi_n^-|^2)\|$.
\end{lemma}

The lemma is an easy estimate from the definition of ${\cal M}$.

The proof of the final statement of the theorem then requires choosing $N$ sufficiently large that the coefficients in the conclusions of the lemma are small enough that
${\cal M}$ is a contraction, 
and summing the Neumann series \eqref{Neum}.
\end{proof}

\subsection{Construction of a second solution and estimates of the product of the two solutions}

Since understanding the asymptotic behavior of the solutions of the perturbed equation requires knowledge of a full set of independent solutions
$\{\phi_n^+, \phi_n^-\}$
to the original equation, we 
recall a standard reduction-of-order formula 
showing that the subdominant solution determines a
second, independent solution, which grows at infinity. 
(E.g., the text \cite{Hart} treats this argument in the continuous case
in \S XI.2, and it can be found in the discrete literature in numerous places, including
\cite{Tr1,ChSh08}.)
The following simple formula 
does not require a subdominant solution, only one that is nonvanishing.  It is true by direct verification that $\psi_n^+$ solves \eqref{one*} and that $W = 1$ when $n=m$.  (Of course it 
is 
derived by positing that 
$\psi_n^+ = \gamma_n \psi_n^-$, substituting, and using the Wronski identity to determine $\gamma_n $.)  

\begin{lemma}\label{secondsoln} (Standard)
Suppose that \eqref{one*} has a solution that is nonzero for all $n$, $n = m, \dots, M$. If
\begin{equation} \label{secondsol}
\psi_n^+ := \begin{cases} 0 & \text{ if } n = m \\  \psi_n^- \ds \sum_{k=m}^{n-1}\frac{1}{\psi_k^- \psi_{k+1}^-} & \text{ if } n>m.
\end{cases}
\end{equation}
then $\psi_n^+$ is an independent solution of \eqref{one*} on the interval 
$[m, M]$.
\end{lemma}

The lemma has some simple but useful consequences:

\begin{cor}\label{boundedcorollary}
Suppose that for some $a$, $|a| > 1$, \eqref{one*} has a solution such that $a^n \psi_n^-  \to 1$ as $n \to \infty$.  Then
\begin{itemize}
\item Every solution $\psi_n$ of \eqref{one*} that is independent of $\psi_n^-$ is exponentially increasing, i.e., $a^{-n} \psi_n \to C \ne 0$ as 
$n \to \infty$.
\item For any solution $\psi_n$, the product $\psi_n \psi_n^-$ is bounded independently of $n$.
\item Given any boundary condition of the form $a \psi_1 + b \psi_2 = 0$, if $0 \notin {\rm sp}( - \Delta + V)$,
then the Green matrix for $(- \Delta + V)$ on $n \ge 1$ is uniformly bounded.
\end{itemize}
\end{cor}

\begin{proof}
Because every solution to \eqref{one*} is a linear combination of $\psi_n^-$ and $\psi_n^+$ as defined in \eqref{secondsol}, it suffices to show the 
first two statements for $\psi_n = \psi_n^+$, which behaves asymptotically like
\be
a^{-n} \sum^{n-1}_k {a^{2 k + 1}}.
\ee
Since this geometric series can be bounded above and below by  ($C_1 + C_2 \, a^{-n} \int_1^{n-1}{a^{2 x +1} dx} = \frac{C_2}{\ln a} a^{n-1} + O(1)$) $C_1 + C_2 a^{n}$ the first two statements follow.

If $0 \ne {\rm sp}(- \Delta + V)$, then the Green matrix is defined, and
\be
G_{mn} = \frac{\psi_{\min(m,n)}^+ \psi_{\max(m,n)}^-}{W[\psi^-, \psi^{+}]},
\ee
where $\psi_n^{+}$ satisfies the boundary condition at $n=1,2$. 
Since this is bounded on any finite set of indices $m,n$, the third statement
follows from the asymptotic estimate of $\psi_n^+$ in the first statement.
\end{proof}

An important case where these estimates apply is captured in the following.

\begin{cor}\label{nearconstantpotential}
Suppose that for some constant $V_\infty \notin [-4, 0]$, $V - V_\infty \in \ell^1$.  Then there is a solution to \eqref{one*} of the type 
$\psi_n^- \sim a^{-n}$ for $|a| > 1$ and an independent solution $\psi_n^+ \sim a^{n}$, and the statements of Corollary~{\rm\ref{boundedcorollary}}
apply.
\end{cor}

\begin{proof}
We can apply Theorem \ref{Banconv} and Corollary \ref{boundedcorollary} once it is observed that
the comparison equation
$$
(- \Delta + V_\infty) \phi_n = 0
$$
has an exponentially decreasing solution, {\it viz}., assuming $V_\infty > 0$, $\phi_n^- = a^{-n}$ for
$a = \frac 1 2 \left(V_\infty +\sqrt{V_\infty^2 + 4 V_\infty}\right)$.  
(The case $V_\infty < -4$ similarly has an exponentially decreasing solution, according to Remark \ref{4 vs 0}.)

\end{proof}

Further conditions for the existence of exponentially decreasing and exponentially increasing solutions may be found in \cite{wong4}.

The existence of a more rapidly decreasing solution
has similar implications:

\begin{cor}\label{superexp}
Suppose that for some $a>1,b>1$, equation \eqref{one*} has a solution such that 
$a^{n^b} \psi_n^-  \to 1$ as $n \to \infty$.  Then
\begin{itemize}
\item Every solution $\psi_n$ of \eqref{one*} that is independent of $\psi_n^-$ 
increases rapidly as $n \to \infty$, but
$a^{-n^b} \psi_n \to 0$ as 
$n \to \infty$.
\item For any solution $\psi_n$, the product $\psi_n \psi_n^-$ is bounded independently of $n$.
\item Given any boundary condition of the form $a \psi_1 + b \psi_2 = 0$, if $0 \ne {\rm sp}(- \Delta + V)$,
then the Green matrix for $(- \Delta + V)$ on $n \ge 1$ is uniformly bounded.
\end{itemize}
\end{cor}

The proof of this is similar to that of Corollary \ref{boundedcorollary}, but details will be left to the interested reader.  On the other hand, 
subdominant solutions that decrease only polynomially fast do not lead to as strong control of the products or of the Green matrix:

\begin{cor}\label{polydecay}
Suppose that for some $a>0$ Equation \eqref{one*} has a solution such that 
$n^a \psi_n^-  \to 1$ as $n \to \infty$.  Then for any solution $\psi$ that is independent 
of $\psi_n^-$, $|\psi_n^- \psi_n| \sim C n$.
\end{cor}
\begin{proofsk}
The argument being familiar from the proof of Corollary \ref{boundedcorollary},
we content ourselves with the application of Formula 
\eqref{secondsol}.  As before, we may as well assume that 
$\psi_n = \psi_n^+$ as defined by that formula, which is of the form
\be
n^{-a} \sum^{n-1}_k {k^{a} (k+1)^a}.
\ee
This is asymptotic to $n^{-a} \int_1^{n-1}{x^{2a} dx} = \frac{n^{{a+1}}}{2 a + 1} + o(n^{a+1})$. The claimed estimate for the product results when this is multiplied by $\psi_n^-$.
\end{proofsk}

Some converse implications, by which the boundedness of $\phi_n^+ \phi_n^-$ controls 
the asymptotic behavior of solutions, will appear in Section \ref{WKB}.

\section{Refined asymptotic estimates}\label{Lilstuff}

Let $\psi$ be a non-trivial solution to the equation $-\Delta + V=0$, when $V^0_n - V_n$ is small, and $\psi_n = a^+_n \phip_n + a^-_n \phim_n$. In this section, we provide a classification of the parameters $a^+_n$ and $a^-_n$  and describe how they converge.

To begin, we prove a preliminary classification of $a^+$ and $a^-$, distinguishing the exceptional cases where the perturbed solutions only depend on one of the comparison solutions in \eqref{BasicRep} for large $n$:

\begin{proposition}[primary classification] Suppose $\beta_n \phip_n \phim_n \to 0$. Then for any non-trivial solution $\psi_n = a^+_n \phip_n + a^-_n \phim_n$ with $a^\pm_n$ satisfying \eqref{e9}, one of the following must be true:
\begin{enumerate}
\item Given any integer $N$, there is an integer $p>N$ such that both $a^+_p$ and $a^-_p$ are non-zero (this will be treated in Theorem~{\rm\ref{dichothm1}} below).
\item There is an integer ${p_0}$ such that
\be
a^+_{{p_0}+k} = a^+_{{p_0}} \neq 0, \quad a^-_{{p_0}+k} = 0  \quad \forall k \geq 0.
\label{case2eq}
\ee
\item There is an integer ${p_0}$ such that
\be
a^-_{{p_0}+k} = a^-_{p_0} \neq 0, \quad a^+_{{p_0}+k} = 0  \quad \forall k \geq 0.
\ee
\end{enumerate}
\label{anonzero}
\end{proposition}

\begin{proof}[Proof of Proposition~{\rm\ref{anonzero}}] Recall the definition in \eqref{e9}. Observe that $\det M_p = \Tr M_p= 0$. Hence, for all $p \in \mathbb{N}$,
\be
\det (I+M_p) = 1 + \Tr M_p + \det M_p = 1, 
\ee which implies that
\be
{\bf a}_n :=\begin{pmatrix} {a_n^+} \\ {a_n^-} \end{pmatrix} = {\bf 0} \text{ for some } m \quad \Leftrightarrow \quad {\bf a}_p = {\bf 0}, \, \forall p \in \mathbb{N} \quad \Leftrightarrow  \quad {\bf a}_1  = {\bf 0}.
\ee

\noindent
Since $\psi$ is not the trivial solution, for any $p$ either $a^+_p \not = 0$ or $a^-_p \not = 0$ (or both). 

Suppose we are not in Case 1. Then there exists $N_0$ such that for all $k \geq N_0$, either $a^+_k=0$ or $a^-_k =0$. Without loss of generality, suppose for some large $p > N_0$, $a^+_p \neq 0$ and $1+ \beta_k \phip_k \phim_k \neq 0$ for all $k \geq p$ (this is possible because $\beta_k \phip_k \phim_k \to 0$). We will show that this corresponds to Case 2.

Observe that
\be
\bpm a^+_{p+1} \\ a^-_{p+1} \epm = (I+M_p ) \bpm a^+_p \\ 0 \epm = a^+_p \bpm 1 + \beta_p \phip_p \phim_p \\ - \beta_p (\phip_p)^2 \epm ,
\label{1111}
\ee which implies $a^+_{p+1} \neq 0$ and as a result, $a^-_{p+1} = 0$. Apply the same argument recursively to obtain \eqref{case2eq}
\end{proof}

\begin{assumptions} Now we focus on Case 1 of Proposition \ref{anonzero}. Without loss of generality, we assume $a^+_1, a^-_1 \neq 0$ and $\sup_n |\beta_n \phip_n \phim_n| < 1$. The latter assumption, together with the assumption in Theorem \ref{dichothm2a} that $\sum_{n=1}^\infty | \beta_n \phip_n \phim_n| < \infty $, guarantees that $p^\pm_n \neq 0$ for all $n$ and that $\lim_{n \to \infty} \Pi^\pm_n \neq 0$.
\label{assump1}
\end{assumptions}

A key observation here is that the recurrence matrix $I+M_n$ in \eqref{e9} can not always be diagonalized, making it impossible to utilize existing techniques in the perturbation theory literature, which heavily relies on the fact that the transfer matrix can be diagonalized (see, e.g., Benzaid--Lutz \cite{BeLu} and the discussion in the Introduction).

A key observation is that the recurrence matrix can be decomposed into the sum of a lower triangular matrix and an upper triangular error matrix:
\begin{align}
I+M_n & =  \bpm 1+\beta_n \phip_n \phim_n & 0 \\ -\beta_n (\phip_n)^2 & 1 - \beta_n \phip_n \phim_n \epm 
+ \bpm     0 & \beta_n (\phim_n)^2 \\ 0 & 0 \epm 
\label{decomposition1} \\
& =:   G_n + E_n \label{Gndef} . 
\end{align}
The advantage of such a decomposition is as follows: let $\Sigma_n$ be defined recursively as
\be
\Sigma_n := \begin{cases} -\beta_n (\phip_n)^2 \ds \prod_{j=1}^{n-1}( 1+ \beta_j \phip_j \phim_j) + (1-\beta_n \phip_n \phim_n) \Sigma_{n-1} , & n \geq 2 ; \\ -\beta_1 (\phip_1)^2 , & n = 1 ,
\end{cases}
\label{Sigmandef}
\ee and let
\be
\Pi^{\pm}_n = \ds \prod_{j=1}^n (1 \pm \beta_j \phip_j \phim_j) := \ds \prod_{j=1}^n p^{\pm}_j.
\label{Pipmndef}
\ee
Under such definitions, we have a closed form for the product $G_n G_{n-1} \dots G_1$:
\be
G_n G_{n-1} \cdots G_1 = G_n \bpm \Pip_{n-1} & 0 \\ \Sigma_{n-1} & \Pim_{n-1} \epm = \bpm \Pip_n & 0 \\ \Sigma_n & \Pim_n \epm .
\ee 

While $G_n$ is an approximation for the recurrence matrix $I+M_n$, we shall show that the 
product $G_n G_{n-1} \cdots G_1$ will serve as an approximation to the actual recurrence relation \eqref{e9} for $a^\pm_n$ under Trench-type conditions on $V^0_n - V_n$:

\begin{theorem}\label{dichothm1}
Let $\phi^{\pm}_n$ be independent solutions to the difference equation $-\Delta + V^0 = 0$. Consider a potential $V$ such that
\be
\ds \sup_n |\Sigma_n| < \infty \text{ and } \ds \sum_{n}^{\infty} \left| (V_n - V^0_n)(\phip_n \phim_n) \right| < \infty.
\ee

If either $\sum_{n=1}^\infty  |\phim_n|^2 < \infty$ or $\sum_{n=1}^\infty |V^0_n - V_n| < \infty$, then any non-trivial solution $\psi$ to the equation $-\Delta + V =0$ can be written as $\psi_n = a^+_n \phip_n + a^-_n \phim_n$ such that one of the following is true:
\begin{enumerate}
\item There is an integer $p$ such that
\be
a^+_{p+k} = a^+_p \neq 0, \quad a^-_{p+k} = 0  \quad \forall k \geq 0.
\ee
\item There is an integer $p$ such that
\be
a^-_{p+k} = a^-_p \neq 0, \quad a^+_{p+k} = 0  \quad \forall k \geq 0.
\ee
\item There exists a constant $f_\infty \not = 0$ such that $a^+_n = \Pi_n^+ f_\infty + o(1)$ and $a^-_n = \Sigma_n f_\infty + o (1)$.
\item There exist constants $f^+_\infty, f^-_\infty$ with $f^-_\infty \neq 0$ such that $a^+_n = \Pi^+_n f^+_\infty + o(1)$ and $a^-_n = \Sigma_n f^+_\infty + f^-_\infty + o(1)$. 
\end{enumerate}
\label{dichothm2a}
\end{theorem}

In Example \ref{constantpotentialexample}, we will construct a potential $V$ such that $\ds \sup_n |\Sigma_n|<\infty$ but $\Sigma_n$ fluctuates as $n$ goes to infinity.

Next, we prove a result such that $|\Sigma_n|$ may go to infinity:
\begin{theorem}\label{dichothm2}
Let the definitions be the same as in Theorem~{\rm\ref{dichothm1}}
 and $V$ be a potential such that
\be
V_n - V^0_n \geq 0 \, (\text{or } \leq 0) \quad \forall n  \in \mathbb{N}  \text{ and } \ds \sum_{n}^{\infty} \left| (V_n - V^0_n)(\phip_n \phim_n) \right| < \infty.
\ee Moreover, we require that $\sup_n |( \phim_n)^2 \Sigma_n| < \infty$.

If $\ds \lim_{n \to \infty} |\Sigma_n |= \infty$, then exactly one of the following must be true: as $n \to \infty$,
\begin{enumerate}
\item $a^+_n = \Pi^+_n(a^+_\infty + o(1) ) \neq 0$ and $a^-_n = \Sigma_{n-1} (a^+_\infty + o(1))$;
\item $a^+_n \to 0$ and $a^-_n \to a^-_\infty \neq 0$.
\end{enumerate}
If $\ds \sup_n |\Sigma_n| < \infty$, the reader may refer to Theorem~{\rm\ref{dichothm1}}. 
\end{theorem}

\begin{proof}[Proof of Theorem \ref{dichothm2a}]
Let $f^+_{n+1}$ and $f^-_{n+1}$ be defined implicitly in \eqref{eq2p2} below:
\be
\bpm a^+_{n+1} \\ a^-_{n+1} \epm = 
\bpm \Pi^+_n & 0 \\ \Sigma_n & \Pi^-_n
\epm \bpm f^+_{n+1} \\ f^-_{n+1} \epm 
\label{eq2p2} .
\ee 

First, we want to prove that
\be
\ds \sup_n |f^+_n| + |f^-_n| < \infty .
\label{fnsumbounded} \ee

Then we prove that exactly one of the following must be true:
\begin{enumerate}
\item $\ds \lim_{n \to \infty} f^+_n {= :} f^+_\infty \neq 0$ and $\ds \lim_{n \to \infty} f^-_n = 0$.
\item $\ds \lim_{n \to \infty} f^+_n {= :} f^+_\infty$ exists and $\ds \lim_{n \to \infty} f^-_n \neq 0$.
\end{enumerate}

To begin, we observe that
\be 
f^+_{n+1} - f^+_n = \ds \frac{a^+_{n+1}}{\Pi_n^+} - \ds \frac{a^+_n}{\Pi_{n-1}^+} = \ds \frac{a^+_{n+1} - p^+_n a^+_n}{\Pi^+_n} = \ds \frac{\beta_n (\phim_n)^2 a^-_n}{\Pi^+_n}  = \ds \frac{\beta_n (\phim_n)^2 (\Sigma_{n-1} f^+_n + \Pi^-_{n-1} f^-_n)}{\Pi^+_n},
\label{fnpineq1}
\ee which, by the triangle inequality, implies that
\be
|f^+_{n+1}| \leq |f^+_n| \left( 1+ \ds \frac{\beta_n (\phim_n)^2 \Sigma_{n-1}}{\Pi^+_n} \right) + |f^-_n| \left| \ds \frac{\beta_n (\phim_n)^2 \Pi^-_{n-1}}{\Pi^+_n} \right| .
\label{fnpineq}
\ee

\be
f^-_{n+1} - f^-_n = \ds \frac{a^-_{n+1} - \Sigma_n f^+_{n+1}}{\Pi^-_n} -  \ds \frac{a^-_{n} - \Sigma_{n-1} f^+_{n}}{\Pi^-_{n-1}} = \ds \frac{a^-_{n+1} - p^-_n a^-_n -\Sigma_n f^+_{n+1} + p^-_n \Sigma_{n-1} f^+_n}{\Pi^-_n}.
\label{fn11}
\ee
By \eqref{e9},
\be
a^-_{n+1} - p^-_n a^-_n = - \beta_n (\phip_n)^2 a^+_n = -\beta_n (\phip_n)^2 \Pi^+_{n-1} f^+_n  = (\Sigma_n - p^-_n \Sigma_{n-1}) f^+_n .
\ee Thus, \eqref{fn11} becomes
\be
f^-_{n+1} - f^-_n = \ds \frac{\Sigma_n (f^+_n - f^+_{n+1})}{\Pi^-_n} = \ds \frac{\beta_n (\phim_n)^2 \Sigma_n ( \Sigma_{n-1} f^+_n + \Pi^-_{n-1} f^-_n)}{\Pi^-_n },
\label{fnmineq1}
\ee which implies that
\be
|f^-_{n+1}| \leq \left| f^-_{n} \right| \left(1 + \left|\ds \frac{\beta_n (\phim_n)^2 \Sigma_n}{p^-_n} \right| \right) + |f^+_n| \left| \ds \frac{\beta_n (\phim_n)^2 \Sigma_n \Sigma_{n-1}}{\Pi^-_n} \right| .
\label{fnmineq}
\ee
Add \eqref{fnpineq} to \eqref{fnmineq}. Since $\sum_{n=1}^\infty |\beta_n \phip_n \phim_n| <\infty$, $\Pi^\pm_n$ converges to a non-zero limit and $p^\pm_n \to 1$. Moreover, $\sup_n |\Sigma_n| < \infty$, so there is a constant $K$ such that
\be
|f^+_{n+1}| + |f^-_{n+1}| \leq (1 + K |\beta_n (\phim_n)^2| ) \left( |f^+_n| + |f^-_n| \right) ,
\ee which implies \eqref{fnsumbounded}.

\begin{dichotomy} There are only two mutually exclusive possibilities for $f^+_n$ and $f^-_n$:
\begin{enumerate}
\item For any pair consisting of an integer $N$ and a constant $M > 0$, there exists an integer $p=p(N,M) > N$ such that
\be
M \left| f^+_{p} \right| >  |{ f^-_p } |.
\label{case1}
\ee
\item There exist an integer $N_0$ and a constant $M_0$ such that
\be
\left| f^+_n \right| \leq M_0 |{f^-_n}|\quad \forall n \geq N_0 .
\label{case2}
\ee 
\end{enumerate}
\end{dichotomy}

Suppose we are in Case 1. Note that \eqref{case1} implies that $f^+_p \neq 0$, because if $f^+_p =0$, then $f^-_p = 0$ which implies $a^+_p = a^-_p = 0$, and hence $a^\pm_n \equiv 0$. This contradicts the assumption that $\psi$ is a non-trivial solution.

Let $p$ be the integer given in \eqref{case1}. We shall specify the choice of $N$ and $M$ later in the proof.

 Let
\be
r_n = \ds \frac{f^-_n}{f^+_n} .
\label{rndef}
\ee 

Note that by the triangle inequality and \eqref{fnpineq1},
\be
\left| \ds \frac{f^+_{p+1}}{f^+_p} \right| \geq 1 - \left| \ds \frac{f^+_{p+1} - f^+_p}{f^+_p} \right| \geq 1 - K_1 |\beta_p (\phim_p)^2| (1 + |r_p| )> 0 .
\label{fpp1}
\ee Thus, $f^+_p \neq 0$ implies $f^+_{p+1} \neq 0$. Furthermore, by inverting \eqref{fpp1}, we obtain
\be
\left| \ds \frac{f^+_p}{f^+_{p+1}} \right| \leq \ds \frac{1}{1 - K_1 |\beta_p (\phim_p)^2 (1+|r_p|)|} < K_2.
\ee

Clearly, both $r_p$ and $r_{p+1}$ are well defined as $f^+_p, f^+_{p+1} \neq 0$. Observe that by the triangle inequality,
\be
|r_{p+1} - r_p | \leq \left| \ds \frac{f^-_{p+1} - f^-_p}{f^+_p} \right| \left| \ds \frac{f^+_p}{f^+_{p+1}} \right| + \left|  \ds \frac{f^+_{p+1} - f^+_p }{f^+_p} \right| \left|\ds \frac{f^-_p}{f^+_p} \right| \left| \frac{f^+_p}{ f^+_{p+1}} \right| .
\label{reqn}
\ee

By by \eqref{fnmineq1}, 
\be
\left| \ds \frac{f^-_{p+1} - f^-_p}{f^+_p} \right| \leq \left| \ds \frac{\beta_p (\phim_p)^2 \Sigma_p}{\Pi^+_p} \right| \left( |\Sigma_{p-1}| +|r_p| |\Pi^-_{p-1} | \right) \leq (1+|r_p|) K_3 |\beta_p (\phim_p)^2| .
\label{fmbound1}
\ee

Similarly, by \eqref{fnpineq1},
\be
\left| \ds \frac{f^+_{p+1} - f^+_p}{f^+_p} \right| \leq   \left| \ds \frac{\beta_p (\phim_p)^2}{\Pi^+_p} \right| \left( |\Sigma_{p-1}| + {|r_p| |\Pi^-_{p-1}|} \right) \leq K_4  |\beta_p (\phim_p)^2| (1+|r_p| ) .
\label{fpbound1}
\ee

By the triangle inequality,
\begin{multline}
|r_{p+1} | \leq |r_{p+1} - r_p| + |r_p| \leq |r_p| \left[1 + \frac{K_5(1+|r_p|) |\beta_p (\phim_p)^2| }{1 - K_1 |\beta_p (\phim_p)^2 (1+|r_p|)|}\right] + \ds \frac{ K_5 |\beta_p (\phim_p)^2|}{1 - K_1 |\beta_p (\phim_p)^2 (1+|r_p|)|} \\ \\ = |r_p| + \frac{(1+|r_p| + |r_p|^2) K_5 |\beta_p (\phim_p)^2| }{1 - K_1 |\beta_p (\phim_p)^2 (1+|r_p|)|} .
\label{rpbound1}
\end{multline}
 In particular, if $|r_p| <1$, then $1+ |r_p| + |r_p|^2 \leq 1 + 2 |r_p|$. Besides, when $p$ is large (which will be the case),
 \be
 1 - K_1 |\beta_p (\phim_p)^2 (1+|r_p|)| > 1 - 2 K_1 |\beta_p (\phim_p)^2| > 1/2.
 \ee Hence, \eqref{rpbound1} becomes
\be
|r_{p+1}| \leq |r_p| + (1 + 2 |r_p|) K_6 |\beta_p (\phim_p)^2| = |r_p| \left(1+ 2 K_6 |\beta_p(\phim_p)^2|  \right) + K_6 |\beta_p(\phim_p)^2| .\label{rpbound2}
\ee

Hence, by an inductive argument we can prove that
\be
|r_{p+k}| \leq \left( \ds \prod_{j=0}^{k-1} \left( 1 + \eta_{p+j} \right)  \right) \left(  |r_{p}|  +  \ds \sum_{j=0}^{k-1} \eta_{p+j} \right) ,
\label{rpk}
\ee where
\be
\eta_k := 2 K_6 |\beta_k (\phim_k)^2| \geq 0.
\label{ekdef} 
\ee 

Since either $\beta_k$ or $|\phim_k|^2$ are summable, $\eta_k \in \ell^1$. Hence, 
\be
P_p: =\ds \prod_{j=0}^{\infty} \left( 1 + \eta_{p+j} \right)  < \ds \prod_{j=0}^{\infty} \left( 1 + \eta_{j} \right) := P_\infty < \infty 
\label{Ppdef}
\ee
and
\be
S_p : = \ds \sum_{j=0}^{k-1} \eta_{p+j} < K_7 \ds \sum_{j=0}^{\infty} \eta_{p+j} \to 0 \text{ as } p \to \infty .
 \label{Spdef}
\ee
Here is how we choose $M$ and $N$: given $\epsilon > 0$, choose $M, N$ such that
\be
|r_p| < M < \ds \frac{\epsilon}{2 P_\infty} \text{ and } \ds \sup_{k \geq N} S_k < \ds \frac{\epsilon}{2 P_\infty} .
\ee
It is guaranteed that there exists an integer $p > N$ such that $|r_p| < M$. By \eqref{rpk},
\be
|r_{p+k}| \leq P_\infty \left( |r_p| + S_p \right) < P_\infty \left( \ds \frac{\epsilon}{2 P_\infty} + \ds \frac{\epsilon}{2 P_\infty} \right) < \epsilon.
\ee
In other words, if \eqref{case1} is true, then $\ds \lim_{n \to \infty} r_n = 0$. Apply this to \eqref{fpbound1}, we get
\be
\left| \ds \frac{f^+_{n+1}}{f^+_n} - 1 \right| \leq K_7 |\beta_n (\phim_n)^2| \in \ell^1.
\label{fpbound2}
\ee 
Since $\log z$ is analytic near $z=1$, in a neighborhood of $1$ there is a 
constant $K_8$, 
arbitrarily close to 1,
such that
\be
|\log z | = |\log z - \log 1| \leq K_8 |z-1|.
\ee
Put $z=f^+_{n+1}/f^+_n$. By \eqref{fpbound2},
\be
\left| \log \ds \frac{f^+_{n+1}}{f^+_n} \right| = O \left( |\beta_n (\phim_n)^2| \right) \in \ell^1. 
\ee Moreover, by the argument following \eqref{fpp1}, we know that there exists an integer $p$ such that $f^+_n \not = 0$ for all $n \geq p$. Therefore,
\be
\log f^+_{n} = \ds \sum_{j=p}^{n-1} \left( \log \ds \frac{ f^+_{j+1} }{ f^+_j }\right) + \log f^+_p .
\ee
That implies the existence of $\lim_{n \to \infty} \log f^+_n$, and
\be
\lim_{n \to \infty}  f^+_n : = f^+_\infty \neq 0 .
\label{fpinfetydef}
\ee Together with the proven fact that $r_n \to 0$, we obtain
\be
\ds \lim_{n \to \infty} f^-_n = 0 .
\ee

Next, suppose \eqref{case2} is true. If $f^-_n \equiv 0$, then $f^+_n \equiv 0$ for all $n \geq N_0$, which contradicts with the assumption that $\psi$ is a non-trivial solution.

Now suppose there is an $m \geq N_0$ such that $f^-_m \neq 0$. Divide both sides of \eqref{fnmineq1} by $f^-_m$. Then we obtain:
\be
\left| \ds \frac{f^-_{m+1}}{f^-_m} - 1\right| \leq K_9 |\beta_m (\phim_m)^2| .
\ee

Using the same argument as in Case 1, we can prove that $f^-_m \neq 0$ implies $f^-_{m+1} \neq 0$, which allows us to apply the same logarithmic argument to prove that
\be
\ds \lim_{n \to \infty} f^-_n := f^-_\infty \neq 0.
\label{fminftydef} 
\ee

By \eqref{fnpineq1},
\be
f^+_{n+1} =  \left( 1 + \beta_n (\phim_n)^2 \ds \frac{\Sigma_{n-1}}{\Pi^+_n} \right)f^+_n  + \left(  \beta_n (\phim_n)^2 \ds \frac{\Pi^-_{n-1}}{\Pi^+_n} \right) f^-_n = (1+O(|\beta_n (\phim_n)^2|) ) f^+_n + O(|\beta_n (\phim_n)^2|) .
\ee Hence,
\be
\ds \lim_{n \to \infty} f^+_n = f^+_\infty \text{ exists. }
\ee 
\end{proof}

\begin{proof}[Proof of Theorem~{\rm\ref{dichothm2}}]  The proof is very similar to the one of Theorem \ref{dichothm1}. The only difference in the proof is that instead of $f^-_n$ and $r_n$ we consider
\be
\hf^-_n:= \ds \frac{f^-_n}{\Sigma_{n-2}} \text{ and } \hat{r}_n = \ds \frac{\hf^-_n}{f^+_n}.
\ee 

In place of $|r_{n+1} - r_n|$ in \eqref{reqn}, we replace it with
\be
\left|\hat{r}_{n+1} -\ds \frac{ \Sigma_{n-2}}{\Sigma_{n-1}} \hat{r}_n \right|
\ee and make use of the fact that $|\Sigma_{n-1}/\Sigma_n| \leq 1$ for all $n$.
\end{proof}

\section{Construction of Comparison Equations}\label{WKB}

In this section we turn to the problem of determining the asymptotic behavior of 
solutions of \eqref{one*} as $n \to \infty$ given a potential $V_n$, where $V_n$ can be either bounded or unbounded.  We shall
construct explicit comparison equations with respect to which we can call upon
the perturbation results of the earlier sections of this article.  
The construction will require a discrete
replacement for the Liouville-Green (familiarly, WKB) approximation, which is a well-known and
quite useful tool for this purpose in the setting of ordinary differential equations \cite{Hart,Olv}.

Our ansatz is that given an equation of the type \eqref{one*}, a related equation is to be sought for which
the solutions are of the form
\begin{equation}\label{prodansatz}
\phi^\pm_n=z_n\prod^n_{\ell=1}S^{\pm 1}_\ell.
\end{equation}
Recall that in the Liouville-Green approximation to 
ordinary differential equations of Schr\"odinger type a comparison is made to a similar equation having a solution basis in the form
$$
V(x)^{-1/4} \exp\left(\pm \int{V(x)^{1/2} dx} \right)
$$
\cite{Olv}.  In common with previous authors, we replace the 
exponential function containing an ``action integral'' by a product
of the quantities we designate $S_n$, but we innovate with
an additional prefactor $z_n$,
to be specified below in \eqref{zdefn}.
This is designed
to bring simplifications in the discrete case 
analogous to those resulting from the prefactor $V(x)^{-1/4}$ in the continuous case.

Before we state the main results of this section, we pause to point out a connection between the factor $z_n$ and the Green matrix for the 
Schr\"odinger operator $- \Delta + \widetilde V$, {\it viz.}, for $n > m$, 
\begin{equation}\label{Greenfnfla}
G_{nm}=\phi_{\min(m,n)}^+ \phi_{\max(m,n)}^- = z_n z_m\prod^n_{\ell=m+1} \frac{1}{S_\ell},
\end{equation}
for which $(- \Delta + \widetilde V) G$ is the identity operator,
by a direct computation.  In case $n=m$, 
\begin{equation}\label{DaHaZ}
G_{nn}=z^2_n.
\end{equation}

In the following section we study the diagonal elements
of the Green matrix and show that they are directly related to the behavior at infinity 
of solutions and to the notion of an Agmon metric.
(cf.\ \cite[Section~4]{DaHa}).  

In order to determine $z_n$, we recall the constancy of the Wronskian of solutions to equations of the type
\eqref{one*}. To simplify the discussion, we take $W=1$, which can be arranged by scaling. Given our assumptions it implies that
\begin{equation}\label{zstuff}
z_nz_{n+1}\left(S_{n+1}-\frac1{S_{n+1}}\right)=1.
\end{equation}

Guided by the case of a constant potential, we expect that if $V_n$ is well-behaved,
then a good choice for $S_n$ is one of the solutions of
$S_n+S_n^{-1}=V_n+2$.
This turns out to be adequate in some bounded cases, but a more sophisticated choice is necessary
when, for example, $V_n$ is allowed to be unbounded.  
We remark that the choice is not unique, because different choices lead to the same asymptotic behavior 
if the comparison potentials they lead to are sufficiently close.
Our discussion will proceed under the supposition that 
$V_n > 0$ for large $n$; 
the case where 
$V_n < -4$ for large $n$ is similar with the systematic sign changes mentioned
in Remark \ref{4 vs 0}.

To determine the best choices for $S_n$ and $z_n$, we consider the equation
\be\label{Sdefn}
S_n + \ds \frac 1 {S_n} = b_n,
\ee
which is effectively a quadratic,
and let $S_n$ be the root of larger magnitude, i.e.,
\be \label{Sfla}
S_n = \ds \frac{b_n + \sqrt{b_n^2 - 4} }{2} , \quad \text{where } |b_n| \geq  2.
\ee

We observe that the relationship
\begin{equation}\label{S diff}
S_n-\frac1{S_n}=\sqrt{b_n^2 - 4}
\end{equation}
necessarily follows.

To be consistent with the Wronski identity \eqref{zstuff} we must set
\begin{equation}\label{zdefn}
z_n:= C_z^{(-1)^n} \sqrt{\ds \frac{(b_{n-1}^2 - 4) (b_{n-3}^2 - 4) \cdots}{(b_n^2 - 4) (b_{n-2}^2 - 4) \cdots} }
\end{equation}
for all $n$, where the constant $C_z$ will be chosen below.
(We clarify that the prefactor simply alternates between $C_z$ and its reciprocal, depending on whether $n$ is even or odd.)

The comparison functions $\phi^{\pm}$ both solve a Schr\"odinger equation with 
potential $\widetilde{V}_n$ given by
\begin{equation}
\widetilde V_n := \frac{\Delta \phi^\pm_n}{\phi^\pm_n}=
\frac{z_{n+1}}{z_n}\, S_{n+1}
+\frac{z_{n-1}}{z_n}\, \frac{1}{S_n} - 2.
\label{wtV}
\end{equation}
(This equation is true by direct substitution for $\phi^+$; to see that is it also true for 
$\phi^-$ requires also substituting from \eqref{zstuff}; cf. a similar argument 
for  \eqref{phiminusstuff} in \S \ref{WKB}.)  We shall in fact show that there is a choice of ways to choose $b_n$ that will lead to a sufficient convergence rate 
of $\widetilde{V}_n - V_n$, 
and that the 
the logarithm of the quantity $S_n$ can be regarded as an Agmon metric
\cite{Agm,HiSi}
controlling the behavior of solutions
$\phi$ of \eqref{one*} at infinity.

For clarity, we first consider the case where $V_n$ is bounded and $V_n \ge C > 0$ for all $n \ge N_0$.  Without loss of generality 
we may assume that $N_0 = 1$, because this does not affect the large-$n$ behavior of a solution basis.  (This simply allows us to avoid choosing phases for some square-roots of quantities that might otherwise not be positive.)

In the case of bounded, slowly varying potentials $V_n$, Theorem \ref{Snestimates bdd} contains estimates for $S_n$ and $z_n$ and uses them to control the  solutions and Green matrix of the comparison equation $(- \Delta + \widetilde{V}) \phi = 0$.  The construction in Theorem \ref{Snestimates bdd} is guided by the special case of a constant potential.

In Theorem \ref{Snestimates2} that follows, we shall present a more general result which covers potentials that are convergent to a finite limit under more relaxed assumptions on $V_n$.

Finally, in Theorem \ref{Snestimates1} we show that the method proposed in this section also works for unbounded potentials that possibly fluctuate.

\begin{theorem}(bounded and slowly varying potential) \label{Snestimates bdd} Suppose that for some $C > 0$, $C \le V_n \in \ell^\infty$,
and that 
$
n(V_{n+1} - V_n) \in \ell^1$. Choose
\be
b_n = b_n^{bdd} := V_n + 2.
\ee
This implies (with a short calculation) that 
\begin{equation}\label{Sdefbdd}
S_n = S_n^{bdd} := \frac 1 2 \left(V_n + 2 + \sqrt{V_n(V_n + 4)}\right).
\end{equation}
The factor $z_n$ is determined by \eqref{zdefn}.  
Then
\item{(a)} 
\be
S_{n+1}^{bdd} - S_n^{bdd} \in \ell^1.
\ee
\item{(b)} $V_n$ converges to a nonzero limit $V_\infty$ as $n \to \infty$, and 
\begin{equation}\label{Czdefn}
C_z := \left(V_\infty (V_\infty + 4)\right)^{-1/4} \prod_{m=1}^{\infty}\sqrt{\frac{V_{2m}(V_{2m}+4)}
{V_{2m-1}(V_{2m-1}+4)}}
\end{equation}
is well defined through a finite convergent product.
\item{(c)} Under this definition of $C_z$ and the one of $z_n$ in \eqref{zdefn}, $z_n z_{n+1} = 1/ \sqrt{V_n(V_n+4)}$ and
\be \label{znlimit}
z_{m} -  \frac{1}{\left[V_\infty (V_\infty+4) \right]^{1/4}} \in \l^1.
\ee 
\item{(d)} The comparison potential ${\wtV}_n$ defined in \eqref{wtV}, satisfies 
$\lim_{n \to \infty}{{\wtV}_n} = V_\infty$, and the Green matrix for $- \Delta + \widetilde{V}$ and the product
$\phi_n^+ \phi_n^-$ are uniformly
bounded.
\item{(e)}
$\widetilde{V}_n -V_n \in \ell^1,$ and therefore, identifying $\widetilde{V}$
with the comparison potential $V^0$ in \eqref{compar}, the solutions of
\eqref{one*} and \eqref{compar} are asymptotically equivalent in the
sense of Theorem \ref{Banconv}.
\end{theorem}

\begin{proof}
The proof for (a) is a direct application of Taylor's Theorem: Following \eqref{Sfla}, we consider the function $f(x) = 1/2(x+2 + \sqrt{x(x+4)} )$, which is differentiable for all $x >0$. In particular, if $x,y>C>0$, $f(y) = f(x) + R(x,y)
$, where $R(x,y)= (y-x) (r+2)/2 \sqrt{r(r+4)} $ for some $r$ between $x,y$, implying that $R(x,y)$ is uniformly bounded in $x,y$ if $x, y >C>0$. Since $S_{n+1}^{bdd}=f(V_{n+1})$ we can write
\be
S_{n+1}^{bdd} = S_n^{bdd} + R(V_n, V_{n+1}) = S_n^{bdd} + O(|V_{n+1} - V_n|) ,
\ee which proves (a).

The fact that if the differences $V_{n+1} - V_n$ are summable, then $V_n$ has a limit is immediate.  To establish the 
convergence of the product \eqref{Czdefn}, let $\delta_m := V_{m} - V_{m-1}$. Then
\be
V_{m}(V_{m} + 4) = V_{m-1}(V_{m-1} + 4) + \delta_m (2 V_{m-1} + \delta_m + 4) 
\ee and since $0 < C \le V_n < V_n + 4$ for all $n$, 
\be
\left| \delta_m \ds \frac{2V_{m-1} + \delta_m + 4}{V_{m-1}(V_{m-1}+4)} \right| 
 = O\left( \left| \delta_m \right| \right).
\ee
As a result,
\be
\frac{V_{m}(V_{m}+4)}{V_{m-1}(V_{m-1}+4)} = 1 + O\left( \left| \ds \frac{V_{m} - V_{m-1}}{V_{m-1}}\right| \right),
\label{errorsummable}
\ee 
which implies that $\ln{\frac{V_{m}(V_{m}+4)}{V_{m-1}(V_{m-1}+4)}} \in \ell^1$.  By taking the logarithm in
\eqref{Czdefn}, the product therefore converges, and is easily seen to be nonzero.
This proves (b).

The same argument for (b) establishes the convergence as $m \to \infty$ of $z_{2m}$ and of $z_{2m+1}$, separately.  The choice of the prefactor in \eqref{Czdefn} ensures that the two limits are the same.  The more precise statement \eqref{znlimit} is
where the assumption that not only $V_{n+1} - V_n$ but also $n(V_{n+1} - V_n) \in \ell^1$ is needed.  From the definition of $z_n$ 
it can be seen (by taking logs and using Taylor's theorem) that $|z_{n+2} - z_n|$ is dominated by a constant times $|V_{n+1}-V_n| + |V_{n}-V_{n-1}|$.  Thus $\sum_m{|z_m - z_\infty|}$ is dominated by a constant times
$$
\sum_{m=1}^{\infty}{\sum_{k=m}^{\infty}{|V_{k}-V_{k-1}|}},
$$ 
which by reversing the order of summation equals
$$
{\sum_{k=1}^{\infty}{(k-1)|V_{k}-V_{k-1}|}} < \infty
$$ 
by assumption.  The other statements in (c) follow by
\eqref{Greenfnfla} and \eqref{DaHaZ}.

Finally, by \eqref{wtV},
\be
R_n := V_n-\widetilde V_n =V_n - \left(\frac{z_{n+1}S_{n+1}^{bdd}}{z_n}+
\frac{z_{n-1}}{z_n S_n^{bdd}}-2\right) = (V_n + 2) - \left(\frac{z_{n+1}S_{n+1}^{bdd}}{z_n}+
\frac{z_{n-1}}{z_n S_n^{bdd}}\right).
\ee
With the aid of \eqref{Sdefbdd},
$$
R_n = (S_{n+1} - S_n) + \left(\frac{z_{n+1}}{z_n} -1 \right)S_{n+1} + \left(\frac{z_{n-1}}{z_n} -1 \right)\frac{1}{S_n}.
$$
Since $S_n$ and $z_n$ both have finite nonzero limits and 
$S_n - \lim_{k \to \infty}{S_k}$ and $z_n - z_\infty$ are both $\ell^1$, each of these three terms is easily seen to belong to 
$\ell^1$.

\end{proof}

\begin{remark} 
The quantity $C_z$ is analogous to the exponential of an action integral in the continuous situation, which shows up in ``tunneling'' effects.
We summarize that in the case where $0 < C \le V_n $ and $V_{n+1} - V_n \in \ell^1$, there is a Liouville-Green basis of comparison 
functions for \eqref{one*}, and The perturbation method of \S \ref{VoC} lets the solutions $\{\psi_n^\pm\}$ be determined from that basis through an iteration that converges for all $n \ge N$ for some finite $N$.  To collect the details in one formula,
the Liouville-Green basis is of the explicit form
\begin{equation}\label{whole LG mess bdd}
\phi_n^\pm = \sqrt{\frac{V_{n-1}(V_{n-1}+4)V_{n-3}(V_{n-3}+4)\dots}{V_{n}(V_{n}+4)V_{n-2}(V_{n-2}+4)\dots}} 
\prod_k^n{\left(\frac{V_k + 2 + \sqrt{V_k(V_k + 4)}}{2}\right)^{\pm 1}}
\end{equation}
(dropping the normalization factors $C_z$ or, resp., $1/C_z$).
\end{remark}

The next result concerns potentials that are convergent to a finite limit under more relaxed assumptions on $V_n$.

\begin{theorem}\label{Snestimates2}(general bounded potential) Let $S_n + 1/S_n = b_n$. where $b_n$ is a bounded function of the potential $V$ such that for all $n$, $C_1 >b_n  > C_2 > 2$ for some constants $C_1, C_2$. If $\sum_n |b_{n+1} - b_n| < \infty$, then
\item{(a)} 
$$
{S_{n+1} - S_n} = O(|b_{n+1} - b_n| ).
$$ 
\item{(b)} 
\begin{equation}\label{Pzdef}
P_M := \prod_{m=1}^{M} \left(\ds \frac{b_{2m}^2-4}{b_{2m-1}^2 - 4} \right) \to P_\infty
\end{equation}
is well defined through a finite convergent product. If $b_n \neq b_{n+1}$ for all $n$, then $P_\infty \neq 0$.

\item{(c)} If $V_{n+1} - V_n \in \ell^1$, we can choose $b_n$ in a way such that $b_{n+1} - b_n \in \ell^1$ and that $\widetilde{V}_n -V_n \in \ell^1$. 
Some appropriate choice will be shown explicitly in the proof below.

\end{theorem}

\begin{proof}[Proof of Theorem \ref{Snestimates2}]
The proof for (a) is a direct application of Taylor's Theorem. Following \eqref{Sfla}, we consider the function $f(x) = 1/2(x + \sqrt{x^2 - 4} )$, which is differentiable for all $x^2 >4$. In particular, if $x,y>2$, $f(y) = f(x) + (y-x) R(x,y)$, where $R(x,y)=  r/2 \sqrt{r^2 - 4} $ for some $r$ between $x,y$, implying that $R(x,y)$ is uniformly bounded in $x,y$ if $x, y >2$. Since $S_{n+1}=f(b_{n+1})$ we can write
\be
S_{n+1} = S_n + (b_{n+1}-b_n) R(b_n, b_{n+1}) = S_n + O(|b_{n+1} - b_n|) ,
\ee which proves (a).

Note that \eqref{wtV1} can be expressed as
\begin{multline}
\widetilde{V}_n  - V_n = \underbrace{
\ds \frac{1}{2} 
 \left[ \ds \frac{b_n}{\sqrt{b_n^2 - 4} } - \ds \frac{b_{n+1}}{\sqrt{b_{n+1}^2 - 4}}  \right] \ds \frac{(b_n^2 - 4) (b_{n-2}^2 - 4) \cdots} {(b_{n-1}^2 - 4) (b_{n-3}^2 - 4) \cdots} }_{(I)}
\\
+
 \underbrace{\left[  \frac {b_{n+1}} {\sqrt{b_{n+1}^2 - 4} } \right] \ds \frac{(b_n^2 - 4) (b_{n-2}^2 - 4) \cdots} {(b_{n-1}^2 - 4) (b_{n-3}^2 - 4) \cdots} - (V_n + 2)}_{(II)}.
\label{wtV4} 
\end{multline} 

Clearly, (I) is summable in $n$ because both $S_n$ and $b_n$ are of bounded variation. Hence, we are left with
\be
\textrm{(II)} = 
 \ds \frac{b_n}{\sqrt{b_n^2 - 4}}\ds \frac{(b_n^2 - 4) (b_{n-2}^2 - 4) \cdots} {(b_{n-1}^2 - 4) (b_{n-3}^2 - 4) \cdots} - (V_n + 2).
\ee

To see what possible choices of $b_n$ that will give us the desired convergence, we let 
\be
J_n = \ds \frac{(b_n^2 - 4) (b_{n-2}^2 - 4) \cdots} {(b_{n-1}^2 - 4) (b_{n-3}^2 - 4) \cdots} .
\label{Jndef}
\ee Then we have
\be
b^2_{n+1} - 4 = J_{n+1} J_{n}
\label{bnjn}
\ee which implies that
\be
\textrm{(II)} = J_n \sqrt{\ds \frac{J_{n+1} J_n + 4}{J_{n+1} J_{n}}} - (V_n+2) = \sqrt{1 + \ds \frac{J_n - J_{n-1}}{J_{n+1}}} \sqrt{(J_{n+1} - J_n)J_n +  J_n^2 + 4} - (V_n+2) .
\label{Jneq1}
\ee

Therefore, a natural choice for  $J_n$ is
\be
J_{n}^2 + 4 = (V_n+2)^2
\label{Jneq2} 
\ee (or equivalently, $J_n = \sqrt{V_n(V_n +4)}$), then $J_{n+1} - J_n = O(|V_{n+1} - V_n|)$ and by \eqref{bnjn}, $b_{n+1} - b_n$ is also $O(|V_{n+1} - V_n|)$ . Under this particular choice of $J_n$,
\be
\sqrt{1 + \ds \frac{J_n - J_{n-1}}{J_{n+1}}} \sqrt{(J_{n+1} - J_n)J_n +  J_n^2 + 4} = (1+ O(V_{n+1} - V_n)) (V_n+2).
\ee 

In fact, there are a number of choices of $J_n$ that we can choose from. By \eqref{wtV4} above, $b_{n+1} - b_n \in \ell^1$ and $J_{n+1} - J_n \in \ell^1$ are are sufficient conditions for $\widetilde{V}_n - V_n \in \ell^1$.

We provide a few examples here for the interested reader:
\item{(a)} (geometric mean) Let $J_{n}^2 + 4 = (V_{n+1}+2) (V_n +2)$. Clearly, under this choice, $J_n^2 + 4 = (V_n + 2)^2 + O(|V_{n+1} - V_n|)$ and by \eqref{bnjn}, $b_{n+1} - b_n = O(|J_{n+1} - J_{n-1}|) =O(|V_{n+1} - V_n| + |V_n -V_{n-1}|) \in \ell^1$ . \\
\item{(b)} (arithmetic mean) Let $J_n^2 + 4 = [(V_n +2)^2 + (V_{n-1} + 2)^2]/2 $.
\\
\item{(c)} (skipping some $V_n$'s) For $k \geq 0$, let $J_{2k}^2+4 = J_{2k+1}^2+4 = (V_{2k}+2)^2$. Then for all $n$, $J_{n+1} -J_n = O(|V_{n+1} - V_{n-1}|) \in \ell^1$. Hence, $b_{k+1} - b_k=O(V_{k+1} - V_{k-2}) \in \ell^1$.

\end{proof}

We now turn to the case where $V_n$ is not bounded.  Here we find that
\emph{the Liouville-Green approximation for the unbounded case simply requires replacing} $V_n+ 2$ 
\emph{by the geometric mean of} $V_n+ 2$ \emph{and} $V_{n-1}+ 2$.  
That is, the canonical choice in the unbounded case is
$$
S_n - \frac{1}{S_n}= \sqrt{(V_n + 2)(V_{n-1} + 2)},
$$
which, we remark, is equivalent to $V_n + 2$ when $V_n$ is bounded and slowly varying.
The argument establishing the accuracy of the Liouville-Green approximation
runs much as in the simpler, more restricted case, but with 
correspondingly more complicated details. As before, this choice of $S_n$ is 
convenient but not unique.

In terms of a general $b_n$ and the $S_n$ related to it according to \eqref{Sdefn}, 
the comparison potential \eqref{wtV} becomes
\be
\widetilde{V}_n  - V_n = 
 \left[ \ds \frac{S_{n+1}}{\sqrt{b_{n+1}^2 - 4}} + \ds \frac{1}{S_n \sqrt{b_n^2 - 4} } \right] \ds \frac{(b_n^2 - 4) (b_{n-2}^2 - 4) \cdots} {(b_{n-1}^2 - 4) (b_{n-3}^2 - 4) \cdots} - (V_n + 2) .
\label{wtV1}
\ee

Note that by \eqref{Sfla} and the fact that $1/S_n = b_n - S_n$,
\begin{eqnarray}
\ds \frac{S_{n+1}}{\sqrt{b_{n+1}^2 - 4}} & = & \ds \frac{b_{n+1}}{2 \sqrt{b_{n+1}^2 - 4}} + \ds \frac 1 2 \\
 \ds \frac{1}{S_n \sqrt{b_n^2 - 4} } & = & \ds \frac{b_n - S_n}{\sqrt{b_n^2 - 4} } = \ds \frac{b_n}{2 \sqrt{b_n^2 - 4}} - \ds \frac{1}{2}.
\end{eqnarray} Hence, \eqref{wtV1} becomes
\be
\widetilde{V}_n - V_n =
 \ds \frac{1}{2}  \left[ \ds \frac{b_{n+1}}{\sqrt{b_{n+1}^2 - 4}} + \ds \frac{b_n}{ \sqrt{b_n^2 - 4} } \right] \ds \frac{(b_n^2 - 4) (b_{n-2}^2 - 4) \cdots} {(b_{n-1}^2 - 4) (b_{n-3}^2 - 4) \cdots} - (V_n + 2) .
\label{wtV2}
\ee

It turns out that the 
convenient choice of $C_z$ in a situation where $V_n$ is unbounded is simply $C_z = 1$.
The theorem reads as follows:

\begin{theorem}\label{Snestimates1}(unbounded potential) Let $V$ be a potential that satisfies
\be
\ds \sum_{n}\frac{1}{V_n^{1/2}} \left( \ds \frac{1}{V_{n+1}^{3/2}  } + \ds \frac{1}{ V_{n-1}^{3/2} }\right) < \infty.
\label{summation}
\ee Let $C_z=1$ and $b_n = S_n + 1/S_n > 0 $ be chosen such that
\be
S_n - \ds \frac 1 {S_n} = \sqrt{b^2_n - 4} = \sqrt{(V_n +2)(V_{n-1}+2)}.
\ee Then 
\item{(a)} $b_n = \sqrt{(V_n+2)(V_{n-1} + 2) + 4}$, $z_n = \frac{1}{\sqrt{V_n+2}}$ and
\be
S_n = \ds \frac{\sqrt{(V_n+2)(V_{n-1}+2)} + \sqrt{4+(V_n+2)(V_{n-1}+2) }}{2} > 1.
\ee
\item{(b)} $\widetilde{V}_n - V_n = O\left( \ds \frac{1}{V_{n+1}^{3/2} V_n^{1/2} }+ \ds \frac{1}{V_n^{1/2} V_{n-1}^{3/2}} \right) \in \ell^1$.
\item{(c)} The Green matrix $G_{m,n}$ is uniformly bounded.
\end{theorem}

\begin{remark} The condition \eqref{summation} is satisfied by unbounded potentials, including 
some that fluctuate.  An example is given \S \ref{example}.
\end{remark}

\begin{proof}[Proof of Theorem \eqref{Snestimates1}] Given this choice of $b^2_n-4$, we have
\be
\ds \frac{(b_n^2 - 4) (b_{n-2}^2 - 4) \cdots} {(b_{n-1}^2 - 4) (b_{n-3}^2 - 4) \cdots} = V_n + 2.
\ee By the definition of $z_n$ in \eqref{zdefn}, this implies (a).

Therefore, by \eqref{wtV2} above,
\be
\widetilde{V}_n - V_n = \ds \frac{V_n+2}{2} \left[ \ds \frac{b_{n+1}}{\sqrt{b_{n+1}^2 - 4}} + \ds \frac{b_n}{ \sqrt{b_n^2 - 4} }  - 2 \right] .
\label{wtV3}
\ee

We apply the relation $\sqrt{a} - \sqrt{b} = (a-b)/(\sqrt{a} + \sqrt{b})$ to
\begin{multline}
\ds \frac{b_{n+1}}{\sqrt{b_{n+1}^2 - 4}} - 1 = \ds \frac{b_{n+1} - \sqrt{b_{n+1}^2 - 4}}{\sqrt{b_{n+1}^2 - 4}} = \ds \frac{b_{n+1}^2 - (b_{n+1}^2 - 4)}{\sqrt{b_{n+1}^2 - 4} \left(b_{n+1} + \sqrt{b_{n+1}^2 - 4} \right)} = \ds \frac{4}{b_{n+1} \sqrt{b_{n+1}^2 - 4} + b_{n+1}^2 - 4 } \\
= \ds \frac{4}{\sqrt{(V_{n+1} + 2)(V_n+2)} \left[(V_{n+1} + 2) (V_n + 2) + 4  \right] + \sqrt{(V_{n+1} + 2)(V_n+2)}  }, 
\end{multline} which is in the order of $O( (V_{n+1}+2)^{3/2} (V_n + 2)^{3/2}$).

Next, we obtain a similar formula for $b_n/\sqrt{b_n^2 - 4}$ and show that it is in the order of $O( (V_n + 2)^{3/2} (V_{n-1}+2)^{3/2} )$. Canceling the $V_n + 2$ term in \eqref{wtV3}, we prove (b).

Statement (c) about the Green matrix and the products of 
comparison solutions then follows from \eqref{Greenfnfla} 
and \eqref{DaHaZ}.
\end{proof}

\section{The diagonal of the Green matrix and 
discrete Agmon distance}\label{discrete DavHar}

We next show that the approximations derived in \S \ref{WKB} for the
solutions and Green matrix of \eqref{one*} by constructing a comparison 
\eqref{compar} are related to exact identities for Green matrices.
In particular, we offer a discrete version of the discovery of Davies and Harrell in
\cite{DaHa}, \S 4, that the diagonal elements $G_{nn}$ of the Green matrix allow 
the full solution space and full Green matrix $G_{mn}$ to be recovered
formulaically.  We build on significant earlier
steps in this direction 
by Chernyavskaya and Shuster \cite{ChSh00,ChSh01}.  As in \cite{DaHa}
we furthermore point out connections between the diagonal of the Green
matrix and an Agmon distance for \eqref{one*}.

First observe that with respect to the
diagonal of the Green function,
there are critical differences between the continuous
Schr\"odinger equations explored in \cite{DaHa} and the discrete equations
considered here.  Consider that, due to Remark \ref{4 vs 0}, if 
\be
G_{mn} = \psi_{\min(m,n)}^+ \psi_{\max(m,n)}^-
\ee
is the Green matrix for some potential function $V_n$, then 
the same diagonal elements $G_{nn}$ also belong to the Green
matrix for an equation of type \eqref{one*} but with potential function
$\widetilde{V}_n = -4 - V_n$.  Thus the uniqueness of the representation 
of \cite{DaHa} is lost, at least to this extent.

We have seen in \S \ref{WKB}
that for the comparison equation solved by the pair of functions
\eqref{prodansatz}, 
the factor $z_n$ equals
the square root of $G_{nn}$.  Meanwhile, if $z_n$
is given, then $S_n$ is determined via \eqref{zstuff}, and 
consequently \eqref{prodansatz} provides a basis for the solution space
of the comparison \eqref{compar}, and \eqref{Greenfnfla} reproduces a full Green
matrix for \eqref{compar}.  
Here we demonstrate that these implications do not rely on
the construction of a comparison equation, but hold in generality
for \eqref{one*}.

Hence let $G_{mn}$ be the Green matrix for any equation
of the form \eqref{one*}, 
and simply define $z_n := \sqrt{G_{nn}}$.  (If $G_{nn}$ is negative, 
a canonical choice of phase could be assigned to $z_n$, but here 
we primarily 
consider the case where $G_{nn} \ge 0$.)
We then use
\eqref{zstuff} to define $S_n$ for $M+1\le n\le N$ {\it viz.}, choosing 
the root analogously to
\eqref{Sfla},
\begin{equation}\label{Sfromz}
S_n^{[z]}:=\frac{1+\sqrt{1+ 4 z_n^2 z_{n-1}^2}}{2z_nz_{n-1}}.
\end{equation}
Here we caution that this choice of the root of the quadratic equation for 
$S_n^{[z]}$ will restrict the possible values of $V_n$ in what follows.
A pair of functions $\varphi^{\pm}_n$ can now be defined by 
the ansatz \eqref{prodansatz},
i.e., when expressed in terms of $z_n$
\begin{equation}\label{phifromz}
\varphi^{\pm}_n:=z_n\prod^n_{k=m+1}\left(\frac{1+\sqrt{1+ 4 z_n^2 z_{n-1}^2}}{2z_nz_{n-1}}\right)^{\pm 1}.
\end{equation}

\noindent
Remarkably, with this definition,
both $\varphi^+$ and $\varphi^-$ solve an equation
of the form \eqref{one*}, where the potential function $V_n$
is determined from $z_n$ via
\begin{align}\label{Vfromz}
V^{[z]}_n&:=\frac{\Delta \varphi^+_n}{\varphi^+_n}=\frac{
1+\sqrt{1+4 z_n^2 z_{n+1}^2}}{2z^2_n}+
\frac{2z^2_{n-1}}
{1+\sqrt{1+4 z_n^2 z_{n-1}^2}} - 2 \nonumber\\
&= \frac{z_{n+1}}{z_n} S_{n+1}^{[z]}+\frac{z_{n-1}}{z_n S_n^{[z]}} - 2,
\end{align}
provided that $V_n > -2$.  (Else a different root must be chosen in \eqref{Sfromz}.)
To see that $\varphi_n^\pm$ solve the same discrete Schr\"odinger equation, 
let us separately calculate
\begin{equation}\label{phiminusstuff}
\frac{\Delta\varphi^-_n}{\varphi^-_n}=\frac{z_{n+1}}{z_n S_{n+1}^{[z]}}+
\frac{z_{n-1}}{z_n}S_n^{[z]} - 2,
\end{equation}
and note that since $S_n^{[z]}$ has been chosen to satisfy 
the equivalent of \eqref{zstuff}, 
the difference between these last 
two expressions is
$$\frac1{z^2_n} - \frac1{z^2_n}=0.$$
This leads to a theorem in the spirit of \cite{DaHa}.

\begin{theorem}\label{DaHa Thm}
Suppose that \eqref{one*} has two independent positive 
solutions for $m\le n\le N$, with $N\ge M+2$, and denote the 
associated Green matrix $G_{mn}$.
Since $G_{nn} > 0$ for $m\le n\le N$, we may
define $z_n := \sqrt{G_{nn}}$.
In terms of $z_n$, define
$S_n^{[z]}$
and
$\varphi^{\pm}_n$
according to
\eqref{Sfromz} and \eqref{phifromz}.
Then
\begin{enumerate}
\item $\varphi^{\pm}_n$ is an independent pair of solutions of \eqref{one*}
for $m<n\le N$.
\item 
$G_{nm}=z_nz_m\prod^n_{\ell=m+1}\frac1{S_\ell^{[z]}}$, $M<m<n\le N$.
\item
The potential function is determined from $G_{nn}$ by a nonlinear difference
equation,
\begin{equation}\label{GtoV}
\frac{1}{2} \left(\sqrt{1 + 4 G_{n\,n}G_{n+1\,n+1}} + \sqrt{1 + 4 G_{n\,n}G_{n-1\,n-1}}\right) 
= (2 + V_n) G_{n\,n}.
\end{equation}
\end{enumerate}
\end{theorem}

\begin{remark}
The assumption that there are two positive solutions is related to the notion of
{\it disconjugacy} in the theory of ordinary differential equations, cf. \cite{Hart,Aga}. 
If, for example, $V_n > 0$ for $n \ge N_0$, then it is 
not difficult to show that no solution can change sign more than once, and that therefore the positivity assumption is satisfied for $n$ sufficiently large.  
As will be seen in the proof, a necessary 
condition for the assumption is that  $V_n > -2$.

Per Remark \ref{4 vs 0} 
the positivity assumption can be replaced by the assumption 
that there are two solutions $\psi_n^\pm$ such that $(-1)^n \psi_n^\pm > 0$.
A sufficient condition for this is that $V_n < -4$ and a necessary condition is
that $V_n < -2$.
\end{remark}

\begin{proof} 
The essential calculation was provided 
in the discussion before the statement
of the theorem.  Given that the Wronskian of $\varphi^-$ and $\varphi^+$ is 1,
these two functions are linearly independent and therefore a basis for the 
solution space of
$$(-\Delta + V^{[z]}_n)\varphi=0.$$
\noindent
Moreover,
$$G_{mn}=\varphi^+_{\min(m,n)}\varphi^-_{\max(m,n)}$$
is a Green function for $-\Delta + V^{[z]}_n$.
The crux of the proof is to show that $V^{[z]}_n$ is the same as the original $V_n$ of
\eqref{one*}.

Because $S_n^{[z]}$ was defined such that
\begin{equation*}
S_{n}^{[z]}-\frac1{S_{n}^{[z]}}=\frac{1}{z_nz_{n-1}}.
\end{equation*}
we may rewrite \eqref{Vfromz} as
\begin{align}\label{Vfromz alt}
2+ V^{[z]}_n&=
\frac{1}{2z^2_n}\left(
\sqrt{1+4 z_n^2 z_{n+1}^2}+
\sqrt{1+4 z_n^2 z_{n-1}^2}\right)
\end{align}
From the definition of $z_n$ and the assumptions of the
theorem, we know that for some independent set of 
positive solutions 
$\psi_n^{\pm}$ of \eqref{one*}, with Wronskian 1, $z_n^2 = \psi_n^+ \psi_n^-$.
Therefore
\begin{align}
4 z_n^2 z_{n \pm 1}^2 &= 4 (\psi_n^+ \psi_{n \pm 1}^-)(\psi_n^- \psi_{n \pm 1}^+)
\nonumber\\
&= (\psi_n^+ \psi_{n \pm 1}^- + \psi_n^- \psi_{n \pm 1}^+)^2 - (\psi_n^+ \psi_{n \pm 1}^- - \psi_n^- \psi_{n \pm 1}^+)^2
\nonumber\\
&=(\psi_n^+ \psi_{n \pm 1}^- + \psi_n^- \psi_{n \pm 1}^+)^2 - 1.\nonumber
\end{align}
Hence \eqref{Vfromz alt} yields
\begin{align}
2+ V^{[z]}_n &= \frac{1}{2 \psi_n^+ \psi_n^-} \left(\psi_n^+ \psi_{n+1}^- + \psi_n^- \psi_{n+1}^+ + \psi_n^+ \psi_{n-1}^- + \psi_n^- \psi_{n-1}^+ \right)\nonumber\\
&= \frac{1}{2 \psi_n^+ \psi_n^-} \left(\psi_n^+ V_n \psi_{n}^- + \psi_n^- V_n \psi_{n}^+ \right)\nonumber\\
&= 2 + V_n,\nonumber
\end{align} 
as claimed, and establishes \eqref{GtoV}, according to \eqref{Vfromz alt}.
\end{proof}

Formula \eqref{phifromz} suggests that $S_n$ can be related to
an Agmon distance \cite{Agm,HiSi}, that is, a metric $d_A(m,n)$
on the positive integer lattice such that every $\ell^2$ solution 
$\phi^-$ of \eqref{one*}
satisfies a bound of the form
$$e^{d_A(0,n)} \phi_n^- \in \ell^\infty,$$ 
and that as a consequence $\phi_n^-$ 
decays rapidly as $n \to \infty$.  Thus if
$z_n$ is bounded we expect
an Agmon distance to be something like
$\sum_{\ell=m+1}^n{\ln S_{\ell}^{[z]}}$. 
(We write the Agmon distance in this way because a metric on the integer lattice 
must be in the form of a sum, as the triangle 
inequality is an equality.) 
In Agmon's theory, however, the distance function should be a quantity that
can be calculated directly from the potential alone, and
indeed, the estimates 
in \S \ref{WKB} already imply some bounds of this form.  
As we shall now see, understanding the diagonal of the Green matrix
allows the derivation
of Agmonish bounds without the
need to control expressions involving $V_{n+1} - V_n$, as in \S \ref{WKB}.  We begin
by showing that 
$G_{nn}$ is comparable to $(V_n + 2)^{-1}$ in a precise sense.

\begin{lemma}\label{G comp 1/V}
Suppose that $\liminf_{n \to \infty}{V_n} > C > 0$ and let 
$G_{m n}$ be any Green matrix for \eqref{one*}.  
Define
$$
K_A := \sqrt{1+\left(\frac{2}{C(C+2)}\right)^2} +  \frac{2}{C(C+2)}.
$$
Then for $n$ sufficiently large,
\begin{equation}\label{Glowerupper}
\frac{1}{V_n + 2} \le G_{nn} \le \frac{K_A}{V_n + 2}.
\end{equation}
Consequently,
\begin{equation}\label{Slowerupper}
\frac{\sqrt{(V_n+2)(V_{n-1}+2)}+ \sqrt{4 +(V_n+2)(V_{n-1}+2)}}{2 K_A}
\le S_n^{[z]} \le 
\frac{\sqrt{(V_n+2)(V_{n-1}+2)}+ \sqrt{4 +(V_n+2)(V_{n-1}+2)}}{2}.
\end{equation}
\end{lemma}

\begin{remark} Note that
the upper bound is of the same form as was found for the Liouville-Green
approximation in (4.36). For a simpler bound $K_A$ could be replaced in these inequalities by
$$
\sqrt{1 + \frac{4}{C^2}} > K_A
$$
(see proof).
\end{remark}

\begin{proof}
The lower bound on $G_{nn}$ is immediate from Statement (3) of
Theorem \ref{DaHa Thm}.
\medskip

\noindent
The upper bound in \eqref{Glowerupper}
requires a spectral estimate.  The Green 
matrix $G_{mn}$ is the kernel of the resolvent operator of a self-adjoint realization of
$- \Delta + V$ on $\ell^2([N, \infty))$ for some $N$, 
where the boundary condition at $n=N,N+1$ is that satisfied by $\varphi_n^+$.  
Since $- \Delta > 0$ on this space (as an operator),
$\inf {\rm sp}(- \Delta + V) > C$, and hence, by the spectral mapping theorem,
$\|(- \Delta + V)^{-1}\|_{\rm op} < C^{-1}$.  Since 
$G_{nn} = \left\langle{e_n, (- \Delta + V)^{-1} e_n }\right\rangle$, 
where $\{e_n\}$ designate the standard unit vectors in $\ell^2$, it follows that
$G_{nn} < C^{-1}$.  Inserting this into
\eqref{Vfromz} would already imply \eqref{Glowerupper} with
$K_A$ replaced by $\sqrt{1+4/C^2}$.  To improve the constant, replace only 
the terms $G_{n\pm1\,n\pm1}$ in \eqref{Vfromz} by $1/C$, getting
\begin{equation}\label{quadrat}
(2 + V_n) \le \frac{\sqrt{1 + \frac{4G_{nn}}{C}}}{G_{nn}}.
\end{equation}
Since
$$
\frac{\sqrt{1 + x y}}{x}
$$
is a decreasing function of $x$ when $x,y > 0$, an upper bound on
$G_{nn}$ is the larger root of \eqref{quadrat} (which is effectively a quadratic).
The claimed upper bound with the constant $K_A$ results 
by keeping one factor $V_n+2$ in the solution of the quadratic, replacing 
the others by $C+2$.

The bounds on $S_n^{[z]}$ result from inserting the bounds on $G_{nn}$ into
\eqref{Sfromz} and collecting terms.
\end{proof}

We can now state some Agmonish bounds.
\begin{cor}\label{Agmon}
Suppose that $\liminf_{n \to \infty}{V_n} > C > 0$ and fix a positive integer $m$.  Then 
the subdominant solution $\varphi^-$ of \eqref{one*} satisfies
\item{(a)}
$$
\left(\prod_{\ell=m}^n{\frac{V_\ell+2}{K_A}}\right)\varphi_n^-  \in \ell^\infty.
$$
\item{(b)}  If, in addition, ${V_{n+1}-V_n} \in \ell^1$, then
$$
\left(\prod_{\ell=m}^n{\frac{V_\ell+2 + \sqrt{V_\ell(V_\ell + 4)}}{2}}\right) \varphi_n^-  \in \ell^\infty.
$$
\end{cor}

\begin{proof}
Recall the representation
\eqref{phifromz}.  Because $z_n$ is bounded, so is
$$
\left(\prod_{\ell}^n{S_\ell^{[z]}}\right) \varphi_n^-.
$$
We then use the lower bound on $S_\ell^{[z]}$ from the lemma, but 
simplify by dropping the 4, which allows the product to telescope in a pleasing way,
producing (a).

For (b) we note that the additional assumption on $V_n$ allows us to conclude that 
$\varphi$ is well-approximated by the Liouville-Green expression in \S \ref{WKB}.  Since 
$z_n$ is again bounded, so is
$$
\left(\prod_{\ell}^n{S_\ell}\right) \varphi_n^-,
$$
using the ansatz \eqref{prodansatz}.  Finally, we recall \eqref{Sfla}.
\end{proof}

\noindent
Thus when
$\liminf_{n \to \infty}{V_n} > 0$,
a suitable Agmon distance $d_A(m,n)$ for \eqref{one*}
is given by 
$$
\sum_{\ell=m+1}^n{\left(\ln(V_l +2) - \ln{K_A}\right)},
$$
or by 
$$
\sum_{\ell=m+1}^n{\ln{\frac{V_\ell+2 + \sqrt{V_\ell(V_\ell + 4)}}{2}}},
$$
provided that $\frac{V_{n+1}-V_n}{V_n} \in \ell^1$.  The latter can be weakened to
the simpler expression
$$
\sum_{\ell=m+1}^n{\ln{(V_\ell+1)}}.
$$

\section{Some illustrative examples}\label{example}
In Theorem \ref{dichothm1}, we consider the problem when $\sup_n |\Sigma_n| < \infty$. Here we construct a potential $V$ such that the boundedness condition of $\Sigma_n$ is satisfied but $\Sigma_n$ fluctuates as $n \to \infty$:
\begin{example}[bounded but fluctuating $\Sigma_n$] Let $V^0_n \equiv V$ such that $V \not \in [-4, 0]$. Then we may find a non-zero $x \in (-1,1)$ such that
\be
x + \ds \frac 1 x = (2+V).
\label{xdef}
\ee The solutions to $-\Delta + V^0 = 0$ are given by
\be
\phim_n = x^n \text{ and } \phip_n = x^{-n} .
\ee and the Wronskian $W$ is $x^{-1} - x$. Consider an asymptotically constant potential:
\be V^{\alpha}_n= V + (-1)^n W x^{2 n} .
\ee In other words, $\beta_n = \beta_n \phip_n \phim_n = (-1)^n x^{2 n}$ is summable and $\beta_n (\phip_n)^2 = (-1)^{n}$. Therefore, $0 < \sup_n |\prod_{j=1}^{n} (1\pm \beta_j \phip_j \phim_j)| <\infty$. For $n \geq 1$,
\be
\Sigma_n =  (-1)^{n+1} \ds \prod_{j=1}^{n-1} (1+ (-1)^j x^{2j}) + (1-(-1)^{n} x^{2n}) \Sigma_{n-1}
\label{f1}
\ee with
\be
\Sigma_1 =1,\quad \Sigma_2 = x^2 - x^4,\quad \Sigma_3 = 1 - x^6 + x^8 - x^{10}.
\ee
Using \eqref{f1}, it is easy to prove that for $k \in \mathbb{N}$,
\be
\ds \lim_{k \to \infty} \Sigma_{2k} = 0 \text{ and } \ds \lim_{k \to \infty} \Sigma_{2k+1} = 1.
\ee
\label{constantpotentialexample}
\end{example}

\begin{example} The main situation we have treated 
is where $V_n \to V_\infty \notin [-4,0]$, with
$V_n - V_\infty \in \ell^2$, for which the solutions are of exponential type, with a 
subdominant solution.  As a second case, let us suppose that
$V_n \to \infty$ 

Eigenfunctions that decay only polynomially are possible when $V_\infty = 0$ or $-4$.  Suppose, for example, that 
$\phi_n^- = n^{- \alpha}$ for some $\alpha > 0$.  This is a solution to a discrete 
Schr\"odinger equation with a potential satisfying
\begin{equation}\label{thrsh1}
V_k  = \frac{(\Delta \phi^-)_k}{\phi_k} = -2 + \left(\frac{k}{k+1}\right)^\alpha + \left(\frac{k}{k-1}\right)^\alpha.
\end{equation}
Using a Taylor expansion, we find that
\begin{equation}\label{thrsh2}
V_k  = \frac{(\Delta \phi^-)_k}{\phi_k} = \frac{\alpha (\alpha+1)}{k^2} + {0}(k^{-4}). 
\end{equation}
Thus polynomial decay can be anticipated when the potential decreases like 
$\gamma k^{-2}$.

\begin{cor}\label{polyexample}
Suppose that for some $\gamma > 0$,
$$
V_k = \frac{\gamma}{k^2} + W_k,
$$
where $k W_k \in \ell^1$.  Then equation \eqref{one*}
has a subdominant solution $\psi_k^-$ such that
\begin{equation}
\lim_{k \to \infty}{k^{\frac 1 2 \left(1 + \sqrt{1 + 4 \gamma}\right)} \psi_k^-} = 1.
\end{equation}
For any solution $\psi_k$ that is linearly independent of $\psi_k^-$,
\begin{equation}
k^{\frac 1 2 \left(1 - \sqrt{1 + 4 \gamma}\right)} \psi_k
\end{equation}
converges to a finite, nonzero value.
\end{cor}

\end{example}

Next, we provide an example such that $|\phip_n \phim_n|$ is not bounded, yet 
the quantity $J_n \in \ell^1$ and therefore
Theorems \ref{dichothm1} and \ref{dichothm2} apply.
\begin{example}[sparse perturbation] Consider a potential $V^0$ such that
\be
V^0_n := \begin{cases}
\ds \frac{2}{(n+1)(n-1)} & n >1 ; \\
-\frac{3}{2} & n =1.
\end{cases}
\ee 
It is easy to verify that $\phi^-_n :=1/n$ is a solution to the equation $-\Delta + V^0 = 0$ under the convention that $\phi^-_{-1} = 0$.
By Corollary \ref{polydecay}, $\phip_n$ obeys $|\phip_n \phi^-_n| \sim C n$ for some constant $C$.

Consider a potential $V$ which is a sparse perturbation of $V^0$:
\be
V_n : = \begin{cases}
V + \ds \frac{W}{n^2} & n = 2^k \text{ for some } k \in \mathbb{N} ;  
\\
V & \text{otherwise}.
\end{cases}
\ee
Under such definitions, $\beta_n \phip_n \phim_n \sim C/n$ is sparsely distributed at powers of $2$ and hence summable.
\end{example}

Finally, we provide an example for which Liouville-Green approximation is accurate, while the potential fluctuates and diverges as 
$n \to \infty$.

\begin{example}\label{fluct ex}
Let $V^a$ be defined such that
\be
V_n^{a} = \begin{cases}
n^a & \text{ if n is odd}; \\ 
1 & \text{ if n is oven},
\end{cases}  \quad a > 2.
\ee Then
\be
\ds \sum_{n}\frac{1}{V_n^{1/2}} \left( \ds \frac{1}{V_{n+1}^{3/2}  } + \ds \frac{1}{ V_{n-1}^{3/2} }\right) 
= \ds \sum_{n \text{ is odd }}  \ds \frac{2}{n^{a/2}} + \ds \sum_{n \text{ is even }} \left( \ds \frac{1}{(n+1)^{3a/2}} + \frac{1}{(n-1)^{3a/2}}\right) < \infty.
\ee
\end{example}

\section*{Appendix: Second-order difference equations and orthogonal polynomials} 

In this section, we will show how the discrete Schr\"odinger operator relates to orthogonal polynomials on the real line.   We begin by recalling some standard facts; the reader may refer to \cite{simon1, simon3} for a comprehensive introduction to the subject.

Let $\mu$ be a non-trivial measure on $\mathbb{R}$ such that for all $n \in \mathbb{N}$, the moments are finite. In other words,
\be
\ds \int_{\mathbb{R}} |x|^n d\mu(x) < \infty .
\ee

We form an inner product and a norm on $L^2(\mathbb{R}, d\mu)$ as follows: for any $f,g \in L^2(\mathbb{R}, d\mu)$, we define an inner product and a norm as follows:
\be
\langle f,g \rangle  =  \ds \int_{\mathbb{R}} \ol{f(x)} g(x) d\mu(x)  , \quad \|f\|^2  =  \ds \int_{\mathbb{R}} f(x)^2 d\mu(x) .
\ee

By the Gram--Schmidt process, we can orthogonalize $1,x, x^2, \dots$ and obtain the family of monic orthogonal polynomial on the real line with respect to the measure $\mu$, which we denote as $(P_n(x))_{n=0}^\infty$. For example, if $\mu= \ds \sqrt{2\pi}^{-1} e^{x^2/2}$, then we obtain the Hermite polynomials; and if $\mu=  \chi_{[-1,1]} dx$, then we obtain the Legendre polynomials.

Let $(p_n(x))_{n=0}^\infty$ denote the family of normalized orthogonal polynomials, i.e., $\|p_n\|^2_\mu = 1$. It is well-known that the monic and the normalized orthonormal polynomials on the real line satisfy the following recurrence relations
\begin{align}
xP_n(x) & = &P_{n+1}(x) + b_{n+1} P_n(x) + a_n^2 P_{n-1}(x) , \label{recurrence1}\\
x p_n(x) & = & a_{n+1}(x) p_{n+1}(x)+ b_{n+1} p_n(x) + a_n p_{n-1}(x) \label{recurrence2} .
\end{align}

Note that \eqref{recurrence2} above can be expressed as follows:
\be
 \bpm b_1  & a_1 & 0 & 0 & \dots \\
a_1 & b_2 & a_2 & 0 & \dots \\
0 & a_2 & b_3 & a_3 & \dots \\
\dots & \dots & \dots & \dots & \dots 
 \epm \bpm 1 \\ p_1(x) \\ p_2(x) \\ \vdots \epm = x \bpm 1 \\ p_1(x) \\ p_2(x) \\ \vdots \epm .
 \label{jacobi}
\ee The tridiagonal matrix in \eqref{jacobi} above is called the \textbf{Jacobi matrix}. The recurrence relation \eqref{recurrence2} can also be expressed in terms of the $2 \times 2$ transfer matrix $A_{n+1}(x)$ as follows
\be
\bpm p_{n+1}(x) \\ a_{n+1} p_n(x) \epm =  \underbrace{ {a_{n+1}}^{-1} \bpm x-b_{n+1} & -1 \\ a_{n+1}^2 & 0 \epm }_{A_{n+1}(x)} \bpm p_n(x) \\ a_n p_{n-1} (x)\epm \, , n\geq 0.
\label{transfer}
\ee

Observe that the discrete Schr\"odinger operator with potential $V$ and energy $E$ on $f$ can be written as
\be
- f_{n+1} - f_{n-1} + (V_n+2) f_n = E f_n .
\label{sch}
\ee Compare \eqref{sch} with \eqref{recurrence2}. Note that the discrete Schr\"odinger equation \eqref{sch} can be seen as having $a_n \equiv 1$ and $b_{n+1} = V_n +2$ and $X=E$. Hence, orthogonal polynomials associated with the measure with recurrence relations $a_n \equiv 1$ and $b_{n+1} = V_n+2$ evaluated at $x=E$ can be seen as a solution of the difference equation \eqref{sch} with initial condition $(p_0(x), a_{0} p_{-1}(x)) = (1,0)$.

The solution to \eqref{transfer} with initial condition $(0,-1)$ (i.e., $n=0$) are known as orthogonal polynomials of the second kind, $(q_n(x))_{n=0}^\infty$, where $q_n(x)$ is a polynomial of degree $n-1$. Therefore, $(p_n(x))_{n=0}^\infty$ and $(q_n(x))_{n=0}^\infty$ form a basis for the solution space of the difference equation \eqref{sch}.

However, for the Schr\"odinger equation, we impose the condition that the solution is square summable (i.e. in $\ell^2(\mathbb{N})$), a property that is not necessarily satisfied by $p_n(E)$. In fact, for any $x_0 \in \mathbb{R}$,
\be
\left( \ds \sum_{k=0}^\infty p_k(x_0)^2 \right)^{-1} = \mu(x_0) .
\ee Hence, $(p_n(E))_n$ is a solution if and only if $E$ is a pure point $\mu$. 

Recall the second-order difference equation \eqref{jeff1} studied by Geronimo--Smith \cite{GeSm} which was briefly discussed in Section \ref{intro}. Note that \eqref{jeff1} can be written in terms of a transfer matrix
\be
\bpm y(n+1) \\ y(n) \epm = d(n+1)^{-1} \bpm  q(n) & -1 \\ d(n+1) & 0  \epm \bpm y(n) \\  y(n-1) \epm
\ee  which resembles the transfer matrix $A_{n+1}(x)$ in \eqref{transfer}. Hence, techniques developed to study the asymptotic behavior of orthogonal polynomials can be applied to study ratio asymptotics of the solutions, which determines whether the limit $\lim_{n \to \infty} y(n+1)/y(n)$ exists and what the limit is in the case that it does. For \eqref{jeff1} and given that $y(n) = \prod_{j=n_0}^n u(j)$, ratio asymptotics means
\be
\ds \frac{y(n+1)}{y(n)} = \ds \frac{u(n+1)}{u(n)},
\ee which explains why it was reasonable for Geronimo--Smith to assume that $\lim_{n \to \infty} u(n+1)/u(n)$ exists should the convergence rates of $q(n)$ and $d(n)$ be sufficiently fast.

For the asymptotic analysis of $p_n(x)$ by means of the transfer matrix when the coefficients are asymptotically identical (meaning $a_n \to a$, $b_n \to b$), the reader may refer to \cite{wong4}.

\section*{Acknowledgments}
We are grateful to Jeff Geronimo, Nina
Chernyavskaya, and Dale Smith for references and helpful remarks.

The second author is supported by the National Science Foundation through NSF-DMS 0456611, M. Lacey, P.I.

\bibliographystyle{plain}

\end{document}